\documentclass[12pt]{article}%
\usepackage{amsfonts}
\usepackage{amsmath}
\usepackage{amssymb}
\usepackage{graphicx}%
\newtheorem{theorem}{Theorem}

\newtheorem{definition}[theorem]{Definition}

\newtheorem{lemma}[theorem]{Lemma}

\newtheorem{remark}[theorem]{Remark}

\newenvironment{proof}[1][Proof]{\noindent\textbf{#1.} }{\ \rule{0.5em}{0.5em}}
\begin{document}

\title{The strict and relaxed stochastic maximum principle for optimal control
problem of backward systems}
\author{Seid Bahlali\thanks{Laboratory of applied mathematics, University Med Khider,
P.O. Box 145, Biskra 07000, Alg\'{e}ria.}}
\date{}
\maketitle

\begin{abstract}
We consider a stochastic control problem where the set of controls is not
necessarily convex and the system is governed by a nonlinear backward
stochastic differential equation. We establish necessary as well as sufficient
conditions of optimality for two models. The first concerns the strict
(classical) controls. The second is an extension of the first to relaxed
controls, who are a measure valued processes.

\ 

\textbf{Keywords}\textit{. }Backward stochastic differential equation,\ strict
control, relaxed control,\textit{\ }maximum principle, adjoint equation,
{variational inequality, variational principle.}

\ 

\textbf{AMS Subject Classification}\textit{. }93 Exx

\end{abstract}

\section{Introduction}

In this paper we study a stochastic control problem where the system is
governed by a nonlinear backward stochastic differential equation (BSDE\ for
short) of the type
\[
\left\{
\begin{array}
[c]{l}%
dy_{t}^{v}=b\left(  t,y_{t}^{v},z_{t}^{v},v_{t}\right)  dt+z_{t}^{v}%
\,dW_{t},\\
y_{T}^{v}=\xi,
\end{array}
\right.
\]
where $b$ is given function, $\xi$ is the terminal data and $W=\left(
W_{t}\right)  _{t\geq0}$ is a standard $d$-dimensional Brownian motion,
defined on a filtered probability space $\left(  \Omega,\mathcal{F},\left(
\mathcal{F}_{t}\right)  _{t\geq0},\mathcal{P}\right)  $ satisfying the usual
conditions. The control variable $v=\left(  v_{t}\right)  $, called strict
(classical) control, is an $\mathcal{F}_{t}$-adapted process with values in
some set $U$ of $\mathbb{R}^{k}$. We denote by $\mathcal{U}$ the class of all
strict controls.

The criteria to be minimized, over the set $\mathcal{U}$, has the form%
\[
J\left(  v\right)  =\mathbb{E}\left[  g\left(  y_{0}^{v}\right)  +%
{\displaystyle\int\nolimits_{0}^{T}}
h\left(  t,y_{t}^{v},z_{t}^{v},v_{t}\right)  dt\right]  ,
\]
where $g$ and $h$ are given maps, and $\left(  y_{t}^{v},z_{t}^{v}\right)  $
is the trajectory of the system controlled by $v$.

A control $u\in\mathcal{U}$ is called optimal if it satisfies%
\[
J\left(  u\right)  =\underset{v\in\mathcal{U}}{\inf}J\left(  v\right)  .
\]

Stochastic control problems for the backward and forward-backward systems have
been studied by many authors. The first contribution of control problems of
forward-backward systems is made by Peng $\left[  30\right]  $, he obtained
the maximum principle with the control domain being convex. Xu $\left[
34\right]  $ established the maximum principle for this kind of problem in the
case where the control domain is not necessary convex, with uncontrolled
diffusion coefficient and a restricted functional cost. The work of Peng
$\left[  30\right]  $ (convex control domain) is generalized by Wu $\left[
33\right]  $, where the system is governed by a fully coupled forward-backward
stochastic differential equation. Shi and Wu $\left[  32\right]  $ extend the
result of Xu $\left[  34\right]  $ to the fully coupled forward-backward
systems, with convex control domain and uncontrolled diffusion coefficient. Ji
and Zhou $\left[  22\right]  $ use the Ekeland variational principle and
establish a maximum principle of controlled forward-backward systems, while
the forward state is constrained in a convex set at the terminal time, and
apply the result to state constrained stochastic linear-quadratic control
models and a recursive utility optimization problem are investigated. All the
cited previous works on stochastic control of forward-backward systems are
obtained by introducing two adjoint equations. In the recent works on the
subject, Bahlali and Labed $\left[  3\right]  $ and Bahlali $\left[  6\right]
$ introduce three adjoint equations to establish necessary as well as
sufficient optimality conditions. In $\left[  3\right]  $ the authors
establish the results in the case where the control domain being nonconvex and
uncontrolled diffusion coefficient. The results of $\left[  6\right]  $, are
obtained while the control domain is convex and with controlled diffusion
coefficient, moreover the author apply his theory to solve the financial model
of cash flow valuation.

On the other hand, stochastic maximum principle of backward systems was
studied by El-Karoui et al $\left[  14\right]  $, where the linear case is
solved and some applications in finance are treated. Dokuchaev and Zhou
$\left[  9\right]  $ established necessary as well as sufficient optimality
conditions, where the control domain is not convex.

Our objective in this paper is to establish necessary as well as sufficient
optimality conditions, of the Pontryagin maximum principle type, for two models.

Firstly, we derive necessary as well as sufficient optimality conditions for
strict controls. Since the set of strict controls is nonconvex, the classical
way to use, is the spike variation method. More precisely, if $u$\ is an
optimal strict control and $v$\ is arbitrary, then with a sufficiently small
$\theta>0$, we define a perturbed control as follows%
\[
u_{t}^{\theta}=\left\{
\begin{array}
[c]{l}%
v\text{ \ \ \ if }t\in\left[  \tau,\tau+\theta\right]  ,\\
u_{t}\text{ \ otherwise.}%
\end{array}
\right.
\]

We then derive the variational equation from the state equation, and the
variational inequality from the fact that%
\[
0\leq J\left(  u^{\theta}\right)  -J\left(  u\right)  .
\]

The major difficulty in doing this is that the state of a backward system and
the functional cost depends on two variables $y_{t}$\ and $z_{t}$. Then, we
can't derive directly the variational inequality, because $z_{t}$\ is hard to
handle, there is no convenient pointwise (in $t$) estimation for it, as
opposed to the first variable $y_{t}$. To overcome this difficulty, we
introduce a new method which consist to transform the initial control problem
to a restricted problem without integral cost, by adding an unidimensional
BSDE. We establish then necessary optimality conditions for the restricted
control problem and by an adequate transformation on the adjoint process and
the adjoint equation associated with the restricted problem, we reformulate
necessary optimality conditions for the initial control problem.

To achieve this part of the paper, we study when these necessary optimality
conditions becomes sufficient.

The second main result in this paper concerns necessary as well as sufficient
optimality conditions for relaxed controls. In the relaxed model, the
controller chooses at time $t$ a probability measure $q_{t}\left(  da\right)
$ on the control set $U$, rather than an element $v_{t}$ of $U$. The system is
then governed by the BSDE%
\[
\left\{
\begin{array}
[c]{l}%
dy_{t}^{q}=\int_{U}b\left(  t,y_{t}^{q},z_{t}^{q},a\right)  q_{t}\left(
da\right)  dt+z_{t}^{q}dW_{t},\\
y_{T}^{q}=\xi.
\end{array}
\right.
\]

The criteria to be minimized, over the set $\mathcal{R}$ of relaxed controls,
has the form%
\[
\mathcal{J}\left(  q\right)  =\mathbb{E}\left[  g\left(  y_{0}^{q}\right)
+\int_{0}^{T}\int_{U}h\left(  t,y_{t}^{q},z_{t}^{q},a\right)  q_{t}\left(
da\right)  dt\right]  .
\]

A control $\mu\in\mathcal{R}$ is called optimal if it satisfies%
\[
\mathcal{J}\left(  \mu\right)  =\inf\limits_{q\in\mathcal{R}}\mathcal{J}%
\left(  q\right)  .
\]

The relaxed control problem is an extension of the previous model of strict
controls. Indeed, if $q_{t}\left(  da\right)  =\delta_{v_{t}}\left(
da\right)  $ is a Dirac measure concentrated at a single point $v_{t}$, then
we get a strict control problem as a particular case of the relaxed one.

By using the Ekeland's variational principle, we are able to establish
necessary optimality conditions for near optimal strict controls converging in
some sense to the relaxed optimal control, by the so called chattering lemma.
The relaxed necessary optimality conditions are then derived by using some
stability properties of the trajectories and the adjoint process with respect
to the control variable.

We note that\ necessary optimality conditions for relaxed controls, where the
systems are governed by a stochastic differential equation, were studied by
Mezerdi and Bahlali $\left[  27\right]  $, Bahlali, Djehiche and Mezerdi
$\left[  4\right]  .$

The paper is organized as follows. In Section 2, we formulate the problem and
give the various assumptions used throughout the paper. Section 3 is devoted
to restrict the initial control problem to a problem without integral cost and
we derive a restricted necessary optimality conditions. In Section 4, we give
our first main result, the necessary optimality conditions for the initial
control problem and under additional hypothesis, we prove that these
conditions becomes sufficient. Finally, in the last Section, we give necessary
optimality conditions for near optimal controls and from this we derive our
second main result in this paper, necessary as well as sufficient optimality
conditions for relaxed controls.

\ 

Along this paper, we denote by $C$ some positive constant, $\mathcal{M}%
_{n\times d}\left(  \mathbb{R}\right)  $ the space of $n\times d$ real matrix
and $\mathcal{M}_{n\times n}^{d}\left(  \mathbb{R}\right)  $ the linear space
of vectors $M=\left(  M_{1},...,M_{d}\right)  $ where $M_{i}\in\mathcal{M}%
_{n\times n}\left(  \mathbb{R}\right)  $. We use the standard calculus of
inner and matrix product.

\section{Formulation of the problem}

Let $\left(  \Omega,\mathcal{F},\left(  \mathcal{F}_{t}\right)  _{t\geq
0},\mathcal{P}\right)  $ be a probability space equipped with a filtration
satisfying the usual conditions, on which a $d$-dimensional Brownian motion
$W=\left(  W_{t}\right)  _{t\geq0}$\ is defined. We assume that $\left(
\mathcal{F}_{t}\right)  $ is the $\mathcal{P}$- augmentation of the natural
filtration of $\left(  W_{t}\right)  _{t\geq0}.$

Let $T$ be a strictly positive real number and $U$ a non empty subset of
$\mathbb{R}^{k}$.

\begin{definition}
\textit{An admissible control is an }$\mathcal{F}_{t}$-\textit{\ adapted
process with values in }$U$\textit{\ such that }
\[
\mathbb{E}\left[  \underset{t\in\left[  0,T\right]  }{\sup}\left\vert
v_{t}\right\vert ^{2}\right]  <\infty.
\]

\textit{We denote by }$\mathcal{U}$\textit{\ the set of all admissible
controls.}
\end{definition}

For any $v\in\mathcal{U}$, we consider the following BSDE
\begin{equation}
\left\{
\begin{array}
[c]{l}%
dy_{t}^{v}=b\left(  t,y_{t}^{v},z_{t}^{v},v_{t}\right)  dt+z_{t}^{v}%
\,dW_{t},\\
y_{T}^{v}=\xi,
\end{array}
\right.
\end{equation}
where
\[
b:\left[  0,T\right]  \times\mathbb{R}^{n}\times\mathcal{M}_{n\times d}\left(
\mathbb{R}\right)  \times U\longrightarrow\mathbb{R}^{n},
\]
and $\xi$ is an $n$-dimensional $\mathcal{F}_{T}$-measurable random variable
such that%
\[
\mathbb{E}\left\vert \xi\right\vert ^{2}<\infty.
\]

The expected cost is defined from $\mathcal{U}$ into $\mathbb{R}$ by%
\begin{equation}
J\left(  v\right)  =\mathbb{E}\left[  g\left(  y_{0}^{v}\right)  +%
{\displaystyle\int\nolimits_{0}^{T}}
h\left(  t,y_{t}^{v},z_{t}^{v},v_{t}\right)  dt\right]  ,
\end{equation}
where%
\begin{align*}
g  &  :\mathbb{R}^{n}\longrightarrow\mathbb{R}\text{,}\\
h  &  :\left[  0,T\right]  \times\mathbb{R}^{n}\times\mathcal{M}_{n\times
d}\left(  \mathbb{R}\right)  \times U\longrightarrow\mathbb{R}\text{.}%
\end{align*}

A control $u\in\mathcal{U}$ is called optimal, if that solves
\begin{equation}
J\left(  u\right)  =\inf\limits_{v\in\mathcal{U}}J\left(  v\right)  .
\end{equation}

Our goal is to establish necessary as well as sufficient optimality conditions
for controls in the form of stochastic maximum principle.

\ 

The following assumptions will be in force throughout this paper
\begin{align}
&  \text{The functions }b,g\text{ and }h\ \text{are continuous in }\left(
y,z,v\right)  \text{, they are }\\
&  \text{differentiable with respect to }\left(  y,z\right)  \text{, and they
derivatives }\nonumber\\
&  b_{y},b_{z},g_{y},h_{y}\text{ and }h_{z}\text{ are continuous in }\left(
y,z,v\right)  \text{ and uniformly bounded.}\nonumber\\
&  b\text{ and }h\text{ are bounded by }C\left(  1+\left\vert y\right\vert
+\left\vert v\right\vert \right)  \text{ and bounded in }z.\nonumber
\end{align}

Under the above hypothesis, for every $v\in\mathcal{U}$, equation $\left(
1\right)  $ has a unique strong $\left(  \mathcal{F}_{t}\right)  _{t}$-adapted
solution and the functional cost $J$ is well defined from $\mathcal{U}$ into
$\mathbb{R}$.

\section{Problem with restricted cost}

Since the function $h$\ of the cost depend explicitly on $z_{t}$, we can't
treat our problem directly. Thus,\textbf{\ } let us in this section restrict
the initial control problem $\left\{  \left(  1\right)  ,\left(  2\right)
,\left(  3\right)  \right\}  $ to a problem without integral cost. For this
end, consider the following unidimensional BSDE%
\[
\left\{
\begin{array}
[c]{l}%
dx_{t}^{v}=h\left(  t,y_{t}^{v},z_{t}^{v},v_{t}\right)  dt+k_{t}^{v}dW_{t},\\
x_{T}^{v}=\eta,
\end{array}
\right.
\]
where $k^{v}$ is an $\left(  1\times d\right)  $ matrix, $\left(  y_{t}%
^{v},z_{t}^{v}\right)  $ is the solution of equation $\left(  1\right)  $ and
$\eta$ is an one-dimensional $\mathcal{F}_{T}$-measurable random variable such
that
\[
\mathbb{E}\left\vert \eta\right\vert ^{2}<\infty.
\]

The above equation admits a unique strong $\left(  \mathcal{F}_{t}\right)
_{t} $- adapted solution.

We put%
\[
\widetilde{y}_{t}=\left(
\begin{array}
[c]{c}%
y_{t}^{v}\\
x_{t}^{v}%
\end{array}
\right)  ,
\]
and consider now the following $\left(  n+1\right)  $-dimensional BSDE%
\begin{equation}
\left\{
\begin{array}
[c]{l}%
d\widetilde{y}_{t}=\widetilde{b}\left(  t,\widetilde{y}_{t},\widetilde{z}%
_{t},v_{t}\right)  dt+\widetilde{z}_{t}dW_{t},\\
\widetilde{y}_{T}=\left(
\begin{array}
[c]{c}%
\xi\\
\eta
\end{array}
\right)  ,
\end{array}
\right.
\end{equation}
where the functions $\widetilde{b}$ is defined from $\left[  0,T\right]
\times\mathbb{R}^{n+1}\times\mathcal{M}_{\left(  n+1\right)  \times d}\left(
\mathbb{R}\right)  \times U$ into $\mathbb{R}^{n+1}$ by%
\[
\widetilde{b}\left(  t,\widetilde{y}_{t},\widetilde{z}_{t},v_{t}\right)
=\left(
\begin{array}
[c]{c}%
b\left(  t,y_{t}^{v},z_{t}^{v},v_{t}\right) \\
h\left(  t,y_{t}^{v},z_{t}^{v},v_{t}\right)
\end{array}
\right)  ,
\]
and $\widetilde{z}_{t}$ is a $\left(  n+1\right)  \times d$ real matrix given
by%
\[
\widetilde{z}_{t}=\left(
\begin{array}
[c]{c}%
z_{t}^{v}\\
k_{t}^{v}%
\end{array}
\right)  =\left(
\begin{array}
[c]{c}%
z_{11}^{v}\ \ z_{12}^{v}\ ...\ z_{1d}^{v}\\
z_{21}^{v}\ \ z_{22}^{v}\ ...\ z_{2d}^{v}\\
\vdots\ \ \ \ \ \ \ \ \ \ \ \ \ \ \ \ \vdots\\
z_{n1}^{v}\ \ z_{n2}^{v}\ ...\ z_{nd}^{v}\\
k_{1}^{v}\ \ k_{2}^{v}\ ...\ k_{d}^{v}%
\end{array}
\right)  .
\]

From $\left(  4\right)  $, $\widetilde{b}$ is uniformly Lipschitz in $\left(
\widetilde{y}_{t},\widetilde{z}_{t}\right)  $, then equation $\left(
1\right)  $ admits a unique strong solution $\left(  \widetilde{y}%
_{t},\widetilde{z}_{t}\right)  $ adapted to the filtration $\left(
\mathcal{F}_{t}\right)  _{t}$.

\ 

Define now the function $\widetilde{g}$ from $\mathbb{R}^{n+1}$ into
$\mathbb{R}$ by%
\[
\widetilde{g}\left(  \widetilde{y}_{t}\right)  =g\left(  y_{t}^{v}\right)
-x_{t}^{v},
\]
and the new functional cost from $\mathcal{U}$ into $\mathbb{R}$ by
\begin{equation}
\widetilde{J}\left(  v\right)  =\mathbb{E}\left[  \widetilde{g}\left(
\widetilde{y}_{0}\right)  \right]  +\mathbb{E}\left[  \eta\right]  .
\end{equation}

It's easy to see that
\[
\widetilde{J}\left(  v\right)  =J\left(  v\right)  .
\]

Consequently, it's sufficient to minimize the restricted cost $\widetilde{J}$
over $\mathcal{U}$. If $u\in\mathcal{U}$ is an optimal solution, that is
\begin{equation}
\widetilde{J}\left(  u\right)  =\inf\limits_{v\in\mathcal{U}}\widetilde
{J}\left(  v\right)  .
\end{equation}

From this transformation, we have reduce our initial problem $\left\{  \left(
1\right)  ,\left(  2\right)  ,\left(  3\right)  \right\}  $ to a new problem
without integral cost. We can now study the restricted problem $\left\{
\left(  5\right)  ,\left(  6\right)  ,\left(  7\right)  \right\}  $ by using a
classical way of spike variation method. We establish necessary optimality
conditions for a restricted problem and by an adequate transformation on the
adjoint process and the adjoint equation associated with the restricted
problem, we reformulate necessary optimality conditions for the initial
control problem $\left\{  \left(  1\right)  ,\left(  2\right)  ,\left(
3\right)  \right\}  $.

\subsection{Preliminary results}

Suppose that $u\in\mathcal{U}$ is an optimal control and denote by $\left(
\widetilde{y}_{t},\widetilde{z}_{t}\right)  $ the solution of $\left(
5\right)  $ corresponding to $u$. Introduce the following perturbation (spike
variation) of the optimal control $u$%
\begin{equation}
u_{t}^{\theta}=\left\{
\begin{array}
[c]{l}%
v\text{ \ \ \ if }t\in\left[  \tau,\tau+\theta\right]  ,\\
u_{t}\text{\ \ otherwise,}%
\end{array}
\right.
\end{equation}
where $0\leq\tau\leq T$ is fixed, $\theta>0$ is sufficiently small and $v$ is
an arbitrary $\mathcal{F}_{t}$-measurable random variable with values in $U$
such that $\mathbb{E}\left[  \left\vert v\right\vert ^{2}\right]  <\infty$.

The control $u^{\theta}$ is admissible and let $\left(  \widetilde{y}%
_{t}^{\theta},\widetilde{z}_{t}^{\theta}\right)  $\ be the solution of
$\left(  5\right)  $ associated with $u_{t}^{\theta}.$

Since $u$ is optimal, the variational inequality will be derived from the fact
that
\begin{equation}
0\leq\widetilde{J}\left(  u^{\theta}\right)  -\widetilde{J}\left(  u\right)  .
\end{equation}

For this end, we need the following lemmas.

\begin{lemma}
\textit{Under assumptions }$\left(  4\right)  $\textit{, we have}%
\begin{align}
\mathbb{E}\left[  \underset{t\in\left[  0,T\right]  }{\sup}\left\vert
\widetilde{y}_{t}^{\theta}-\widetilde{y}_{t}\right\vert ^{2}\right]   &  \leq
C\theta^{2},\\
\mathbb{E}%
{\displaystyle\int\nolimits_{0}^{T}}
\left\vert \widetilde{z}_{t}^{\theta}-\widetilde{z}_{t}\right\vert ^{2}dt  &
\leq C\theta^{2}.
\end{align}

\end{lemma}

\begin{proof}
We have%
\[
\left\{
\begin{array}
[c]{ll}%
d\left(  \widetilde{y}_{t}^{\theta}-\widetilde{y}_{t}\right)  = & \left[
\widetilde{b}\left(  t,\widetilde{y}_{t}^{\theta},\widetilde{z}_{t}^{\theta
},u_{t}^{\theta}\right)  -\widetilde{b}\left(  t,\widetilde{y}_{t}%
,\widetilde{z}_{t}^{\theta},u_{t}^{\theta}\right)  \right]  dt\\
& \left[  \widetilde{b}\left(  t,\widetilde{y}_{t},\widetilde{z}_{t}^{\theta
},u_{t}^{\theta}\right)  -\widetilde{b}\left(  t,\widetilde{y}_{t}%
,\widetilde{z}_{t},u_{t}^{\theta}\right)  \right]  dt\\
& \left[  \widetilde{b}\left(  t,\widetilde{y}_{t},\widetilde{z}_{t}%
,u_{t}^{\theta}\right)  -\widetilde{b}\left(  t,\widetilde{y}_{t}%
,\widetilde{z}_{t},u_{t}\right)  \right]  dt\\
& +\left(  \widetilde{z}_{t}^{\theta}-\widetilde{z}_{t}\right)  dW_{t},\\
\left(  \widetilde{y}_{T}^{\theta}-\widetilde{y}_{T}\right)  = & 0.
\end{array}
\right.
\]

Put%
\begin{align*}
Y_{t}^{\theta}  &  =\widetilde{y}_{t}^{\theta}-\widetilde{y}_{t},\\
Z_{t}^{\theta}  &  =\widetilde{z}_{t}^{\theta}-\widetilde{z}_{t},
\end{align*}
and%
\begin{align}
\varphi^{\theta}\left(  t,Y_{t}^{\theta},Z_{t}^{\theta}\right)   &  =%
{\displaystyle\int\nolimits_{0}^{1}}
\widetilde{b}_{y}\left(  t,\widetilde{y}_{t}+\lambda\left(  \widetilde{y}%
_{t}^{\theta}-\widetilde{y}_{t}\right)  ,\widetilde{z}_{t}+\lambda\left(
\widetilde{z}_{t}^{\theta}-\widetilde{z}_{t}\right)  ,u_{t}^{\theta}\right)
Y_{t}^{\theta}d\lambda\\
&  +%
{\displaystyle\int\nolimits_{0}^{1}}
\widetilde{b}_{z}\left(  t,\widetilde{y}_{t}+\lambda\left(  \widetilde{y}%
_{t}^{\theta}-\widetilde{y}_{t}\right)  ,\widetilde{z}_{t}+\lambda\left(
\widetilde{z}_{t}^{\theta}-\widetilde{z}_{t}\right)  ,u_{t}^{\theta}\right)
Z_{t}^{\theta}d\lambda\nonumber\\
&  +\widetilde{b}\left(  t,\widetilde{y}_{t},\widetilde{z}_{t},u_{t}^{\theta
}\right)  -\widetilde{b}\left(  t,\widetilde{y}_{t},\widetilde{z}_{t}%
,u_{t}\right)  .\nonumber
\end{align}

Then%
\begin{equation}
\left\{
\begin{array}
[c]{l}%
dY_{t}^{\theta}=\varphi^{\theta}\left(  t,Y_{t}^{\theta},Z_{t}^{\theta
}\right)  dt+Z_{t}^{\theta}dW_{t},\\
Y_{T}^{\theta}=0.
\end{array}
\right.
\end{equation}

The above equation is a linear BSDE with bounded coefficients and with
terminal condition $Y_{T}^{\theta}=0$. Then by applying a priori estimates
(see Briand et al $\left[  8,\ \text{Proposition 3.2, Page 7}\right]  $), we
get%
\[
\mathbb{E}\left[  \underset{t\in\left[  0,T\right]  }{\sup}\left\vert
Y_{t}^{\theta}\right\vert ^{2}+%
{\displaystyle\int\nolimits_{0}^{T}}
\left\vert Z_{t}^{\theta}\right\vert ^{2}dt\right]  \leq C\mathbb{E}%
\left\vert
{\displaystyle\int\nolimits_{0}^{T}}
\left\vert \varphi^{\theta}\left(  t,0,0\right)  \right\vert dt\right\vert
^{2}.
\]

From $\left(  12\right)  $, we get%
\[
\mathbb{E}\left[  \underset{t\in\left[  0,T\right]  }{\sup}\left\vert
Y_{t}^{\theta}\right\vert ^{2}+%
{\displaystyle\int\nolimits_{0}^{T}}
\left\vert Z_{t}^{\theta}\right\vert ^{2}dt\right]  \leq C\mathbb{E}%
\left\vert
{\displaystyle\int\nolimits_{0}^{T}}
\left\vert \widetilde{b}\left(  t,\widetilde{y}_{t},\widetilde{z}_{t}%
,u_{t}^{\theta}\right)  -\widetilde{b}\left(  t,\widetilde{y}_{t}%
,\widetilde{z}_{t},u_{t}\right)  \right\vert dt\right\vert ^{2}.
\]

By the definition of $u^{\theta}$, we have%
\begin{align*}
\mathbb{E}\left[  \underset{t\in\left[  0,T\right]  }{\sup}\left\vert
Y_{t}^{\theta}\right\vert ^{2}+%
{\displaystyle\int\nolimits_{0}^{T}}
\left\vert Z_{t}^{\theta}\right\vert ^{2}dt\right]   &  \leq C\mathbb{E}%
\left\vert
{\displaystyle\int\nolimits_{\tau}^{\tau+\theta}}
\left\vert \widetilde{b}\left(  t,\widetilde{y}_{t},\widetilde{z}%
_{t},v\right)  -\widetilde{b}\left(  t,\widetilde{y}_{t},\widetilde{z}%
_{t},u_{t}\right)  \right\vert dt\right\vert ^{2}\\
&  \leq C\mathbb{E}\left\vert \underset{t\in\left[  0,T\right]  }{\sup
}\left\vert \widetilde{b}\left(  t,\widetilde{y}_{t},\widetilde{z}%
_{t},v\right)  -\widetilde{b}\left(  t,\widetilde{y}_{t},\widetilde{z}%
_{t},u_{t}\right)  \right\vert
{\displaystyle\int\nolimits_{\tau}^{\tau+\theta}}
dt\right\vert ^{2}.
\end{align*}

By $\left(  4\right)  $, $b$ is with linear growth with respect to $\left(
y,v\right)  $ and bounded in $z$, then $\widetilde{b}$ satisfy the same
properties, and we get%
\[
\mathbb{E}\left[  \underset{t\in\left[  0,T\right]  }{\sup}\left\vert
Y_{t}^{\theta}\right\vert ^{2}+%
{\displaystyle\int\nolimits_{0}^{T}}
\left\vert Z_{t}^{\theta}\right\vert ^{2}dt\right]  \leq C\mathbb{\theta}%
^{2}.
\]

The lemma is proved.
\end{proof}

\subsection{Necessary optimality conditions for restricted problem}

We can now state necessary optimality conditions for a restricted control
problem $\left\{  \left(  5\right)  ,\left(  6\right)  ,\left(  7\right)
\right\}  . $

\begin{theorem}
(necessary optimality conditions for restricted problem) \textit{Let }$\left(
u,\widetilde{y},\widetilde{z}\right)  $\textit{\ be an optimal solution of the
restricted control problem }$\left\{  \left(  5\right)  ,\left(  6\right)
,\left(  7\right)  \right\}  $\textit{. Then there exists a unique adapted
process}
\[
\widetilde{p}\in\mathcal{L}^{2}\left(  \left[  0,T\right]  ;\mathbb{R}%
^{n+1}\right)  ,
\]
\textit{which is solution of the following forward stochastic differential
equation}%
\begin{equation}
\left\{
\begin{array}
[c]{l}%
-d\widetilde{p}_{t}=\widetilde{H}_{y}\left(  t,\widetilde{y}_{t},\widetilde
{z}_{t},\widetilde{p}_{t},u_{t}\right)  dt+\widetilde{H}_{z}\left(
t,\widetilde{y}_{t},\widetilde{z}_{t},\widetilde{p}_{t},u_{t}\right)
dW_{t},\\
\widetilde{p}_{0}=\widetilde{g}_{y}\left(  \widetilde{y}_{0}\right)  ,
\end{array}
\right.
\end{equation}
\textit{such that }
\begin{equation}
\widetilde{H}\left(  t,\widetilde{y}_{t},\widetilde{z}_{t},\widetilde{p}%
_{t},u_{t}\right)  =\underset{v\in U}{\max}\widetilde{H}\left(  t,\widetilde
{y}_{t},\widetilde{z}_{t},\widetilde{p}_{t},v\right)  \;;\;a.e\;,\;a.s,
\end{equation}
where the Hamiltonian $\widetilde{H}$ is defined from $\left[  0,T\right]
\times\mathbb{R}^{n+1}\times\mathcal{M}_{\left(  n+1\right)  \times d}\left(
\mathbb{R}\right)  \times\mathbb{R}^{n+1}\times U$ into $\mathbb{R}$ by%
\[
\widetilde{H}\left(  t,\widetilde{y}_{t},\widetilde{z}_{t},\widetilde{p}%
_{t},u_{t}\right)  =\widetilde{b}\left(  t,\widetilde{y}_{t},\widetilde{z}%
_{t},u_{t}\right)  \widetilde{p}_{t}.
\]

\end{theorem}

\begin{proof}
For simplicit, we put%
\[
\Lambda_{t}^{\theta}=\left(  t,\widetilde{y}_{t}+\lambda\left(  \widetilde
{y}_{t}^{\theta}-\widetilde{y}_{t}\right)  ,\widetilde{z}_{t}+\lambda\left(
\widetilde{z}_{t}^{\theta}-\widetilde{z}_{t}\right)  ,u_{t}^{\theta}\right)
.
\]

Since $u$ minimizes the cost $\widetilde{J}$ over $\mathcal{U}$, then
\end{proof}

\begin{align*}
0  &  \leq\widetilde{J}\left(  u^{\theta}\right)  -\widetilde{J}\left(
u\right) \\
&  \leq\mathbb{E}\left[  \widetilde{g}\left(  \widetilde{y}_{0}^{\theta
}\right)  -\widetilde{g}\left(  \widetilde{y}_{0}\right)  \right] \\
&  \leq\mathbb{E}%
{\displaystyle\int\nolimits_{0}^{1}}
\widetilde{g}_{y}\left[  \widetilde{y}_{0}+\lambda\left(  \widetilde{y}%
_{0}^{\theta}-\widetilde{y}_{0}\right)  \right]  \left(  \widetilde{y}%
_{0}^{\theta}-\widetilde{y}_{0}\right)  d\lambda\\
&  \leq\mathbb{E}\left[  \widetilde{g}_{y}\left(  \widetilde{y}_{0}\right)
\left(  \widetilde{y}_{0}^{\theta}-\widetilde{y}_{0}\right)  \right]
+\mathbb{E}%
{\displaystyle\int\nolimits_{0}^{1}}
\left[  \widetilde{g}_{y}\left(  \widetilde{y}_{0}+\lambda\left(
\widetilde{y}_{0}^{\theta}-\widetilde{y}_{0}\right)  \right)  -\widetilde
{g}_{y}\left(  \widetilde{y}_{0}\right)  \right]  \left(  \widetilde{y}%
_{0}^{\theta}-\widetilde{y}_{0}\right)  d\lambda.
\end{align*}

We remark from $\left(  14\right)  $ that%
\[
\widetilde{p}_{0}=\widetilde{g}_{y}\left(  \widetilde{y}_{0}\right)  .
\]

Then%
\begin{equation}
0\leq\mathbb{E}\left[  \widetilde{p}_{0}\left(  \widetilde{y}_{0}^{\theta
}-\widetilde{y}_{0}\right)  \right]  +\mathbb{E}%
{\displaystyle\int\nolimits_{0}^{1}}
\left[  \widetilde{g}_{y}\left(  \widetilde{y}_{0}+\lambda\left(
\widetilde{y}_{0}^{\theta}-\widetilde{y}_{0}\right)  \right)  -\widetilde
{g}_{y}\left(  \widetilde{y}_{0}\right)  \right]  \left(  \widetilde{y}%
_{0}^{\theta}-\widetilde{y}_{0}\right)  d\lambda.
\end{equation}

By applying It\^{o}'s formula to $\widetilde{p}_{t}\left(  \widetilde{y}%
_{t}^{\theta}-\widetilde{y}_{t}\right)  $, we get%
\begin{align*}
\mathbb{E}\left[  \widetilde{p}_{0}\left(  \widetilde{y}_{0}^{\theta
}-\widetilde{y}_{0}\right)  \right]   &  =\mathbb{E}%
{\displaystyle\int\nolimits_{0}^{T}}
{\displaystyle\int\nolimits_{0}^{1}}
\left[  \widetilde{b}_{y}\left(  \Lambda_{t}^{\theta}\right)  -\widetilde
{b}_{y}\left(  t,\widetilde{y}_{t},\widetilde{z}_{t},u_{t}\right)  \right]
\left(  \widetilde{y}_{t}-\widetilde{y}_{t}^{\theta}\right)  \widetilde{p}%
_{t}d\lambda dt\\
&  +\mathbb{E}%
{\displaystyle\int\nolimits_{0}^{T}}
{\displaystyle\int\nolimits_{0}^{1}}
\left[  \widetilde{b}_{z}\left(  \Lambda_{t}^{\theta}\right)  -\widetilde
{b}_{z}\left(  t,\widetilde{y}_{t},\widetilde{z}_{t},u_{t}\right)  \right]
\left(  \widetilde{z}_{t}-\widetilde{z}_{t}^{\theta}\right)  \widetilde{p}%
_{t}d\lambda dt\\
&  +\mathbb{E}%
{\displaystyle\int\nolimits_{0}^{T}}
\left[  \widetilde{b}\left(  t,\widetilde{y}_{t},\widetilde{z}_{t}%
,u_{t}\right)  -\widetilde{b}\left(  t,\widetilde{y}_{t},\widetilde{z}%
_{t},u_{t}^{\theta}\right)  \right]  \widetilde{p}_{t}dt.
\end{align*}

Then $\left(  16\right)  $ becomes%
\begin{align}
0  &  \leq\mathbb{E}%
{\displaystyle\int\nolimits_{0}^{T}}
\left[  \widetilde{H}\left(  t,\widetilde{y}_{t},\widetilde{z}_{t}%
,\widetilde{p}_{t},u_{t}\right)  -\widetilde{H}\left(  t,\widetilde{y}%
_{t},\widetilde{z}_{t},\widetilde{p}_{t},u_{t}^{\theta}\right)  \right]  dt\\
&  -\mathbb{E}%
{\displaystyle\int\nolimits_{0}^{1}}
\left[  \widetilde{g}_{y}\left(  \widetilde{y}_{0}+\lambda\left(
\widetilde{y}_{0}^{\theta}-\widetilde{y}_{0}\right)  \right)  -\widetilde
{g}_{y}\left(  \widetilde{y}_{0}\right)  \right]  \left(  \widetilde{y}%
_{t}-\widetilde{y}_{t}^{\theta}\right)  d\lambda\nonumber\\
&  +\mathbb{E}%
{\displaystyle\int\nolimits_{0}^{T}}
{\displaystyle\int\nolimits_{0}^{1}}
\left[  \widetilde{b}_{y}\left(  \Lambda_{t}^{\theta}\right)  -\widetilde
{b}_{y}\left(  t,\widetilde{y}_{t},\widetilde{z}_{t},u_{t}\right)  \right]
\left(  \widetilde{y}_{t}-\widetilde{y}_{t}^{\theta}\right)  \widetilde{p}%
_{t}d\lambda dt\nonumber\\
&  +\mathbb{E}%
{\displaystyle\int\nolimits_{0}^{T}}
{\displaystyle\int\nolimits_{0}^{1}}
\left[  \widetilde{b}_{z}\left(  \Lambda_{t}^{\theta}\right)  -\widetilde
{b}_{z}\left(  t,\widetilde{y}_{t},\widetilde{z}_{t},u_{t}\right)  \right]
\left(  \widetilde{z}_{t}-\widetilde{z}_{t}^{\theta}\right)  \widetilde{p}%
_{t}d\lambda dt\nonumber
\end{align}

Let us show that
\begin{equation}
\mathbb{E}%
{\displaystyle\int\nolimits_{0}^{T}}
{\displaystyle\int\nolimits_{0}^{1}}
\left[  \widetilde{b}_{y}\left(  \Lambda_{t}^{\theta}\right)  -\widetilde
{b}_{y}\left(  t,\widetilde{y}_{t},\widetilde{z}_{t},u_{t}\right)  \right]
\left(  \widetilde{y}_{t}-\widetilde{y}_{t}^{\theta}\right)  \widetilde{p}%
_{t}d\lambda dt\leq C\theta^{3/2},
\end{equation}
and%
\begin{equation}
\mathbb{E}%
{\displaystyle\int\nolimits_{0}^{T}}
{\displaystyle\int\nolimits_{0}^{1}}
\left[  \widetilde{b}_{z}\left(  \Lambda_{t}^{\theta}\right)  -\widetilde
{b}_{z}\left(  t,\widetilde{y}_{t},\widetilde{z}_{t},u_{t}\right)  \right]
\left(  \widetilde{z}_{t}-\widetilde{z}_{t}^{\theta}\right)  \widetilde{p}%
_{t}d\lambda dt\leq C\theta^{3/2}.
\end{equation}

Indeed, by using the Cauchy-Schwartz inequality to term in the left hand side
of $\left(  19\right)  $, we get%
\begin{align*}
&  \mathbb{E}%
{\displaystyle\int\nolimits_{0}^{T}}
{\displaystyle\int\nolimits_{0}^{1}}
\left[  \widetilde{b}_{z}\left(  \Lambda_{t}^{\theta}\right)  -\widetilde
{b}_{z}\left(  t,\widetilde{y}_{t},\widetilde{z}_{t},u_{t}\right)  \right]
\left(  \widetilde{z}_{t}-\widetilde{z}_{t}^{\theta}\right)  \widetilde{p}%
_{t}d\lambda dt\\
&  \leq\left(  \mathbb{E}%
{\displaystyle\int\nolimits_{0}^{T}}
{\displaystyle\int\nolimits_{0}^{1}}
\left\vert \left[  \widetilde{b}_{z}\left(  \Lambda_{t}^{\theta}\right)
-\widetilde{b}_{z}\left(  t,\widetilde{y}_{t},\widetilde{z}_{t},u_{t}\right)
\right]  \widetilde{p}_{t}\right\vert ^{2}d\lambda dt\right)  ^{1/2}\left(
\mathbb{E}%
{\displaystyle\int\nolimits_{0}^{T}}
\left\vert \widetilde{z}_{t}-\widetilde{z}_{t}^{\theta}\right\vert
^{2}dt\right)  ^{1/2}.
\end{align*}

By $\left(  11\right)  $, we obtain
\begin{align*}
&  \mathbb{E}%
{\displaystyle\int\nolimits_{0}^{T}}
{\displaystyle\int\nolimits_{0}^{1}}
\left[  \widetilde{b}_{z}\left(  \Lambda_{t}^{\theta}\right)  -\widetilde
{b}_{z}\left(  t,\widetilde{y}_{t},\widetilde{z}_{t},u_{t}\right)  \right]
\left(  \widetilde{z}_{t}-\widetilde{z}_{t}^{\theta}\right)  \widetilde{p}%
_{t}d\lambda dt\\
&  \leq C\theta\left(  \mathbb{E}%
{\displaystyle\int\nolimits_{0}^{T}}
{\displaystyle\int\nolimits_{0}^{1}}
\left\vert \left[  \widetilde{b}_{z}\left(  \Lambda_{t}^{\theta}\right)
-\widetilde{b}_{z}\left(  t,\widetilde{y}_{t},\widetilde{z}_{t},u_{t}\right)
\right]  \widetilde{p}_{t}\right\vert ^{2}d\lambda dt\right)  ^{1/2}.
\end{align*}

By the definition of $u^{\theta}$, we have%
\begin{align*}
&  \mathbb{E}%
{\displaystyle\int\nolimits_{0}^{T}}
{\displaystyle\int\nolimits_{0}^{1}}
\left[  \widetilde{b}_{z}\left(  \Lambda_{t}^{\theta}\right)  -\widetilde
{b}_{z}\left(  t,\widetilde{y}_{t},\widetilde{z}_{t},u_{t}\right)  \right]
\left(  \widetilde{z}_{t}-\widetilde{z}_{t}^{\theta}\right)  \widetilde{p}%
_{t}d\lambda dt\\
&  \leq C\theta\left(  \mathbb{E}%
{\displaystyle\int\nolimits_{\tau}^{\tau+\theta}}
{\displaystyle\int\nolimits_{0}^{1}}
\left\vert \left[  \widetilde{b}_{z}\left(  \Lambda_{t}^{v}\right)
-\widetilde{b}_{z}\left(  t,\widetilde{y}_{t},\widetilde{z}_{t},u_{t}\right)
\widetilde{p}_{t}\right]  \right\vert ^{2}d\lambda dt\right)  ^{1/2}.
\end{align*}

Since $\widetilde{b}_{y}$ is bounded, we get%
\begin{align*}
&  \mathbb{E}%
{\displaystyle\int\nolimits_{0}^{T}}
{\displaystyle\int\nolimits_{0}^{1}}
\left[  \widetilde{b}_{z}\left(  \Lambda_{t}^{\theta}\right)  -\widetilde
{b}_{z}\left(  t,\widetilde{y}_{t},\widetilde{z}_{t},u_{t}\right)  \right]
\left(  \widetilde{z}_{t}^{\theta}-\widetilde{z}_{t}\right)  \widetilde{p}%
_{t}d\lambda dt\\
&  \leq C\theta\left(
{\displaystyle\int\nolimits_{\tau}^{\tau+\theta}}
\mathbb{E}\left\vert \widetilde{p}\right\vert ^{2}dt\right)  ^{1/2}.
\end{align*}

Since $\widetilde{p}\in\mathcal{L}^{2}\left(  \left[  0,T\right]
;\mathbb{R}^{n+1}\right)  ,$ we obtain%
\begin{align*}
&  \mathbb{E}%
{\displaystyle\int\nolimits_{0}^{T}}
{\displaystyle\int\nolimits_{0}^{1}}
\left[  \widetilde{b}_{z}\left(  \Lambda_{t}^{\theta}\right)  -\widetilde
{b}_{z}\left(  t,\widetilde{y}_{t},\widetilde{z}_{t},u_{t}\right)  \right]
\left(  \widetilde{z}_{t}^{\theta}-\widetilde{z}_{t}\right)  \widetilde{p}%
_{t}d\lambda dt\\
&  \leq\left(  C%
{\displaystyle\int\nolimits_{\tau}^{\tau+\theta}}
dt\right)  ^{1/2}C\theta=C\theta^{3/2}.
\end{align*}

Relation $\left(  19\right)  $ is proved.

$\left(  18\right)  $ is proved by the same method and by using $\left(
10\right)  $ and the fact that $b_{y}$ is bounded.

\ 

Now, by $\left(  17\right)  $, $\left(  18\right)  $ and $\left(  19\right)  $
we get
\begin{align*}
0  &  \leq\mathbb{E}%
{\displaystyle\int\nolimits_{0}^{T}}
\left[  \widetilde{H}\left(  t,\widetilde{y}_{t},\widetilde{z}_{t}%
,\widetilde{p}_{t},u_{t}\right)  -\widetilde{H}\left(  t,\widetilde{y}%
_{t},\widetilde{z}_{t},\widetilde{p}_{t},u_{t}^{\theta}\right)  \right]  dt\\
&  +\mathbb{E}%
{\displaystyle\int\nolimits_{0}^{1}}
\left[  \widetilde{g}_{y}\left(  \widetilde{y}_{0}+\lambda\left(
\widetilde{y}_{0}^{\theta}-\widetilde{y}_{0}\right)  \right)  -\widetilde
{g}_{y}\left(  \widetilde{y}_{0}\right)  \right]  \left(  \widetilde{y}%
_{0}^{\theta}-\widetilde{y}_{0}\right)  d\lambda\\
&  +C\theta^{3/2}.
\end{align*}

By applying the Cauchy-Schwartz inequality to the second term in the right
hand side of the above inequality, we get%
\begin{align*}
0  &  \leq\mathbb{E}%
{\displaystyle\int\nolimits_{0}^{T}}
\left[  \widetilde{H}\left(  t,\widetilde{y}_{t},\widetilde{z}_{t}%
,\widetilde{p}_{t},u_{t}\right)  -\widetilde{H}\left(  t,\widetilde{y}%
_{t},\widetilde{z}_{t},\widetilde{p}_{t},u_{t}^{\theta}\right)  \right]  dt\\
&  +\left(
{\displaystyle\int\nolimits_{0}^{1}}
\mathbb{E}\left\vert \widetilde{g}_{y}\left(  \widetilde{y}_{0}+\lambda\left(
\widetilde{y}_{0}^{\theta}-\widetilde{y}_{0}\right)  \right)  -\widetilde
{g}_{y}\left(  \widetilde{y}_{0}\right)  \right\vert ^{2}d\lambda\right)
^{1/2}\left(  \mathbb{E}\left\vert \widetilde{y}_{0}^{\theta}-\widetilde
{y}_{0}\right\vert ^{2}\right)  ^{1/2}\\
&  +C\theta^{3/2}.
\end{align*}

By $\left(  10\right)  $, we deduce%
\begin{align*}
0  &  \leq\mathbb{E}%
{\displaystyle\int\nolimits_{0}^{T}}
\left[  \widetilde{H}\left(  t,\widetilde{y}_{t},\widetilde{z}_{t}%
,\widetilde{p}_{t},u_{t}\right)  -\widetilde{H}\left(  t,\widetilde{y}%
_{t},\widetilde{z}_{t},\widetilde{p}_{t},u_{t}^{\theta}\right)  \right]  dt\\
&  +C\theta\left(
{\displaystyle\int\nolimits_{0}^{1}}
\mathbb{E}\left\vert \widetilde{g}_{y}\left(  \widetilde{y}_{0}+\lambda\left(
\widetilde{y}_{0}^{\theta}-\widetilde{y}_{0}\right)  \right)  -\widetilde
{g}_{y}\left(  \widetilde{y}_{0}\right)  \right\vert ^{2}d\lambda\right)
^{1/2}\\
&  +C\theta^{3/2}.
\end{align*}

From the definition of $u_{t}^{\theta}$, we have%
\begin{align*}
0  &  \leq\mathbb{E}%
{\displaystyle\int\nolimits_{\tau}^{\tau+\theta}}
\left[  \widetilde{H}\left(  t,\widetilde{y}_{t},\widetilde{z}_{t}%
,\widetilde{p}_{t},u_{t}\right)  -\widetilde{H}\left(  t,\widetilde{y}%
_{t},\widetilde{z}_{t},\widetilde{p}_{t},v\right)  \right]  dt\\
&  +C\theta\left(
{\displaystyle\int\nolimits_{0}^{1}}
\mathbb{E}\left\vert \widetilde{g}_{y}\left(  \widetilde{y}_{0}+\lambda\left(
\widetilde{y}_{0}^{\theta}-\widetilde{y}_{0}\right)  \right)  -\widetilde
{g}_{y}\left(  \widetilde{y}_{0}\right)  \right\vert ^{2}d\lambda\right)
^{1/2}\\
&  +C\theta^{3/2}.
\end{align*}

Dividing by $\theta$, we get%
\begin{align}
0  &  \leq\frac{1}{\theta}\mathbb{E}%
{\displaystyle\int\nolimits_{\tau}^{\tau+\theta}}
\left[  \widetilde{H}\left(  t,\widetilde{y}_{t},\widetilde{z}_{t}%
,\widetilde{p}_{t},u_{t}\right)  -\widetilde{H}\left(  t,\widetilde{y}%
_{t},\widetilde{z}_{t},\widetilde{p}_{t},v\right)  \right]  dt\\
&  +C\left(
{\displaystyle\int\nolimits_{0}^{1}}
\mathbb{E}\left\vert \widetilde{g}_{y}\left(  \widetilde{y}_{0}+\lambda\left(
\widetilde{y}_{0}^{\theta}-\widetilde{y}_{0}\right)  \right)  -\widetilde
{g}_{y}\left(  \widetilde{y}_{0}\right)  \right\vert ^{2}d\lambda\right)
^{1/2}\nonumber\\
&  +C\theta^{1/2}.\nonumber
\end{align}

Since $\widetilde{g}_{y}$ is continuous and bounded, then by $\left(
10\right)  $ and the dominated convergence theorem, we have%
\[
\underset{\theta\rightarrow0}{\lim}C%
{\displaystyle\int\nolimits_{0}^{1}}
\left(  \mathbb{E}\left\vert \widetilde{g}_{y}\left(  \widetilde{y}%
_{0}+\lambda\left(  \widetilde{y}_{0}^{\theta}-\widetilde{y}_{0}\right)
\right)  -\widetilde{g}_{y}\left(  \widetilde{y}_{0}\right)  \right\vert
^{2}\right)  ^{1/2}d\lambda=0.
\]

Then, by taking the limit as $\theta\rightarrow0$ in $\left(  20\right)  $, we
obtain%
\[
0\leq\underset{\theta\rightarrow0}{\lim}\frac{1}{\theta}\mathbb{E}%
{\displaystyle\int\nolimits_{\tau}^{\tau+\theta}}
\left[  \widetilde{H}\left(  t,\widetilde{y}_{t},\widetilde{z}_{t}%
,\widetilde{p}_{t},u_{t}\right)  -\widetilde{H}\left(  t,\widetilde{y}%
_{t},\widetilde{z}_{t},\widetilde{p}_{t},v\right)  \right]  dt.
\]

This implies that%
\[
0\leq\mathbb{E}\left[  \widetilde{H}\left(  \tau,\widetilde{y}_{\tau
},\widetilde{z}_{\tau},\widetilde{p}_{\tau},u_{\tau}\right)  -\widetilde
{H}\left(  \tau,\widetilde{y}_{\tau},\widetilde{z}_{\tau},\widetilde{p}_{\tau
},v\right)  \right]  ,\ d\tau-a.e.
\]

Now, let $a\in U$ be a deterministic element and $F$ be an arbitrary element
of the $\sigma$-algebra $\mathcal{F}_{t}$, and set%
\[
w_{t}=a\mathbf{1}_{F}+u_{t}\mathbf{1}_{%
\Omega
-F}.
\]

It is obvious that $w$ is an admissible control.

Since $0\leq\tau\leq T$, then for every bounded $U$-valued, $\mathcal{F}_{t}%
$-measurable random variable $v$ such that $\mathbb{E}|v|^{2}<+\infty$, we get%
\[
0\leq\mathbb{E}\left[  \widetilde{H}\left(  t,\widetilde{y}_{t},\widetilde
{z}_{t},\widetilde{p}_{t},u_{t}\right)  -\widetilde{H}\left(  t,\widetilde
{y}_{t},\widetilde{z}_{t},\widetilde{p}_{t},v\right)  \right]  ,\ dt-a.e,
\]

Applying the above inequality with $w$, we get
\[
0\leq\mathbb{E}[\mathbf{1}_{F}(\widetilde{H}\left(  t,\widetilde{y}%
_{t},\widetilde{z}_{t},\widetilde{p}_{t},u_{t}\right)  -\widetilde{H}\left(
t,\widetilde{y}_{t},\widetilde{z}_{t},\widetilde{p}_{t},a\right)  )],\ \forall
F\in\mathcal{F}_{t},
\]

which implies that%
\[
0\leq\mathbb{E}[\widetilde{H}\left(  t,\widetilde{y}_{t},\widetilde{z}%
_{t},\widetilde{p}_{t},u_{t}\right)  -\widetilde{H}\left(  t,\widetilde{y}%
_{t},\widetilde{z}_{t},\widetilde{p}_{t},a\right)  \ /\ \mathcal{F}_{t}].
\]

The quantity inside the conditional expectation is $\mathcal{F}_{t}%
$-measurable, and thus the result follows immediately. This prove theorem 3.

\section{Necessary and sufficient optimality conditions for strict controls}

Starting from the results of the last section, we can now reformulate the
restricted necessary optimality conditions given by theorem $3$, and state
necessary as well as sufficient optimality conditions for the initial control
problem $\left\{  \left(  1\right)  ,\left(  2\right)  ,\left(  3\right)
\right\}  .$

\subsection{Necessary optimality conditions}

\begin{theorem}
(necessary optimality conditions for strict controls) \textit{Let }$\left(
u,y^{u},z^{u}\right)  $\textit{\ be an optimal solution of the initial control
problem }$\left\{  \left(  1\right)  ,\left(  2\right)  ,\left(  3\right)
\right\}  $\textit{. Then there exists a unique adapted processes}%
\[
p^{u}\in\mathcal{L}^{2}\left(  \left[  0,T\right]  ;\mathbb{R}^{n}\right)  ,
\]
\textit{which are solution of the following forward stochastic differential
equation}%
\begin{equation}
\left\{
\begin{array}
[c]{l}%
-dp_{t}^{u}=H_{y}\left(  t,y_{t}^{u},z_{t}^{u},p_{t}^{u},u_{t}\right)
dt+H_{z}\left(  t,y_{t}^{u},z_{t}^{u},p_{t}^{u},u_{t}\right)  dW_{t},\\
p_{0}^{u}=g_{y}\left(  y_{0}^{u}\right)  ,
\end{array}
\right.
\end{equation}
such that%
\begin{equation}
H\left(  t,y_{t}^{u},z_{t}^{u},p_{t}^{u},u_{t}\right)  =\underset{v\in U}%
{\max}H\left(  t,y_{t}^{u},z_{t}^{u},p_{t}^{u},v\right)  \;;\;a.e\;,\;a.s,
\end{equation}
where the Hamiltonian $H$ is defined from $\left[  0,T\right]  \times
\mathbb{R}^{n}\times\mathcal{M}_{n\times d}\left(  \mathbb{R}\right)
\times\mathbb{R}^{n}\times U$ into $\mathbb{R}$ by%
\[
H\left(  t,y,z,p,v\right)  =pb\left(  t,y,z,v\right)  -h\left(
t,y,z,v\right)  .
\]

\end{theorem}

\begin{proof}
We put%
\[
\widetilde{p}_{t}=\left(
\begin{array}
[c]{c}%
p_{t}^{u}\\
-1
\end{array}
\right)  .
\]

From the definition of $\widetilde{H},\ \widetilde{p},\ \widetilde{b}$ and
$\widetilde{z}$, we have%
\begin{equation}
\widetilde{H}\left(  t,\widetilde{y}_{t},\widetilde{z}_{t},\widetilde{p}%
_{t},u_{t}\right)  =H\left(  t,y_{t}^{u},z_{t}^{u},p_{t}^{u},u_{t}\right)  ,
\end{equation}
and from the adjoint equation $\left(  14\right)  $, we can easily deduce
$\left(  21\right)  $. Finally $\left(  22\right)  $ is derived immediately
from $\left(  23\right)  $ and $\left(  15\right)  $.
\end{proof}

\subsection{Sufficient optimality conditions}

\begin{theorem}
(Sufficient optimality conditions for strict controls). If we assume that, $U$
is convex and for every $v\in\mathcal{U}$ and for all $t\in\left[  0,T\right]
$, the function $g$ is convex and $\left(  y_{t},z_{t},v_{t}\right)
\longrightarrow H\left(  t,y_{t},z_{t},p_{t},v_{t}\right)  $ is concave. Then
$u$ is an optimal control of the problem $\left\{  \left(  1\right)  ,\left(
2\right)  ,\left(  3\right)  \right\}  $ if it satisfies $\left(  22\right)
.$
\end{theorem}

\begin{proof}
Let $u$ be an arbitrary admissible control (candidate to be optimal) and
$\left(  y_{t}^{u},z_{t}^{u}\right)  $ the solution of $\left(  1\right)  $
associated with $u$. For any admissible control $v$, with associated
trajectory $\left(  y_{t}^{v},z_{t}^{v}\right)  $, we have%
\begin{align*}
J\left(  v\right)  -J\left(  u\right)   &  =\mathbb{E}\left[  g\left(
y_{0}^{v}\right)  -g\left(  y_{0}^{u}\right)  \right] \\
&  +\mathbb{E}%
{\displaystyle\int\nolimits_{0}^{T}}
\left[  h\left(  t,y_{t}^{v},z_{t}^{v},v_{t}\right)  -h\left(  t,y_{t}%
^{u},z_{t}^{u},u_{t}\right)  \right]  dt.
\end{align*}

Since $g$ is convex, then%
\[
g\left(  y_{0}^{v}\right)  -g\left(  y_{0}^{u}\right)  \geq g_{y}\left(
y_{0}^{u}\right)  \left(  y_{0}^{v}-y_{0}^{u}\right)  .
\]

Then%
\begin{align*}
J\left(  v\right)  -J\left(  u\right)   &  \geq\mathbb{E}\left[  g_{y}\left(
y_{0}^{u}\right)  \left(  y_{0}^{v}-y_{0}^{u}\right)  \right] \\
&  +\mathbb{E}%
{\displaystyle\int\nolimits_{0}^{T}}
\left[  h\left(  t,y_{t}^{v},z_{t}^{v},v_{t}\right)  -h\left(  t,y_{t}%
^{u},z_{t}^{u},u_{t}\right)  \right]  dt.
\end{align*}

We remark from $\left(  21\right)  $ that
\[
p_{0}^{u}=g_{y}\left(  y_{0}^{u}\right)  .
\]

Then, we have%
\begin{align*}
J\left(  v\right)  -J\left(  u\right)   &  \geq\mathbb{E}\left[  p_{0}%
^{u}\left(  y_{0}^{v}-y_{0}^{u}\right)  \right] \\
&  +\mathbb{E}%
{\displaystyle\int\nolimits_{0}^{T}}
\left[  h\left(  t,y_{t}^{v},z_{t}^{v},v_{t}\right)  -h\left(  t,y_{t}%
^{u},z_{t}^{u},u_{t}\right)  \right]  dt.
\end{align*}

By applying It\^{o}'s formula to $p_{t}^{u}\left(  y_{t}^{v}-y_{t}^{u}\right)
$, we obtain%
\begin{align*}
&  J\left(  v\right)  -J\left(  u\right) \\
&  \geq\mathbb{E}%
{\displaystyle\int\nolimits_{0}^{T}}
\left[  H_{y}\left(  t,y_{t}^{u},z_{t}^{u},p_{t}^{u},u_{t}\right)  \left(
y_{t}^{v}-y_{t}^{u}\right)  +H_{z}\left(  t,y_{t}^{u},z_{t}^{u},p_{t}%
^{u},u_{t}\right)  \left(  z_{t}^{v}-z_{t}^{u}\right)  \right]  dt\\
&  +\mathbb{E}%
{\displaystyle\int\nolimits_{0}^{T}}
\left[  H\left(  t,y_{t}^{u},z_{t}^{u},p_{t}^{u},u_{t}\right)  -H\left(
t,y_{t}^{v},z_{t}^{v},p_{t}^{u},v_{t}\right)  \right]  dt.
\end{align*}

Since $H$ is concave in $\left(  y,z,u\right)  $, then%
\begin{align*}
&  H\left(  t,y_{t}^{v},z_{t}^{v},p_{t}^{u},v_{t}\right)  -H\left(
t,y_{t}^{u},z_{t}^{u},p_{t}^{u},u_{t}\right) \\
&  \leq H_{y}\left(  t,y_{t}^{u},z_{t}^{u},p_{t}^{u},u_{t}\right)  \left(
y_{t}^{v}-y_{t}^{u}\right) \\
&  +H_{z}\left(  t,y_{t}^{u},z_{t}^{u},p_{t}^{u},u_{t}\right)  \left(
z_{t}^{v}-z_{t}^{u}\right)  +H_{v}\left(  t,y_{t}^{u},z_{t}^{u},p_{t}%
^{u},u_{t}\right)  \left(  v_{t}-u_{t}\right)  .
\end{align*}

Or equivalently%
\begin{align*}
&  H_{v}\left(  t,y_{t}^{u},z_{t}^{u},p_{t}^{u},u_{t}\right)  \left(
u_{t}-v_{t}\right) \\
&  \leq H\left(  t,y_{t}^{u},z_{t}^{u},p_{t}^{u},u_{t}\right)  -H\left(
t,y_{t}^{v},z_{t}^{v},p_{t}^{u},v_{t}\right) \\
&  +H_{y}\left(  t,y_{t}^{u},z_{t}^{u},p_{t}^{u},u_{t}\right)  \left(
y_{t}^{v}-y_{t}^{u}\right)  +H_{z}\left(  t,y_{t}^{u},z_{t}^{u},p_{t}%
^{u},u_{t}\right)  \left(  z_{t}^{v}-z_{t}^{u}\right)  .
\end{align*}

Then, we get%
\begin{equation}
J\left(  v\right)  -J\left(  u\right)  \geq\mathbb{E}%
{\displaystyle\int\nolimits_{0}^{T}}
H_{v}\left(  t,y_{t}^{u},z_{t}^{u},p_{t}^{u},u_{t}\right)  \left(  u_{t}%
-v_{t}\right)  dt.
\end{equation}

We know that $H\left(  t,y_{t}^{u},z_{t}^{u},p_{t}^{u},.\right)  $ is concave,
then $-H\left(  t,y_{t}^{u},z_{t}^{u},p_{t}^{u},.\right)  $ is convex from $U$
into $\mathbb{R}$. Furthermore $U$ is convex and $-H\left(  t,y_{t}^{u}%
,z_{t}^{u},p_{t}^{u},.\right)  $ is continuous, G\^{a}teaux-differentiable,
with differential continuous, then from the convex optimization principle (see
Ekeland-Temam $\left[  11,\text{\ prop 2.1, page 35}\right]  $), we have%
\[
-H\left(  t,y_{t}^{u},z_{t}^{u},p_{t}^{u},u_{t}\right)  =\underset{v_{t}\in
U}{\inf}-H\left(  t,y_{t}^{u},z_{t}^{u},p_{t}^{u},v_{t}\right)
\Longleftrightarrow-H_{v}\left(  t,y_{t}^{u},z_{t}^{u},p_{t}^{u},u_{t}\right)
\left(  v_{t}-u_{t}\right)  \geq0.
\]

Or equivalently%
\[
H\left(  t,y_{t}^{u},z_{t}^{u},p_{t}^{u},u_{t}\right)  =\underset{v_{t}\in
U}{\max}H\left(  t,y_{t}^{u},z_{t}^{u},p_{t}^{u},v_{t}\right)
\Longleftrightarrow H_{v}\left(  t,y_{t}^{u},z_{t}^{u},p_{t}^{u},u_{t}\right)
\left(  u_{t}-v_{t}\right)  \geq0.
\]

Then from the necessary condition of optimality $\left(  22\right)  $, we
deduce that
\[
H_{v}\left(  t,y_{t}^{u},z_{t}^{u},p_{t}^{u},u_{t}\right)  \left(  u_{t}%
-v_{t}\right)  \geq0.
\]

And from $\left(  24\right)  $, we have
\[
J\left(  v\right)  -J\left(  u\right)  \geq0.
\]

The theorem is proved.
\end{proof}

\section{The relaxed model}

In this section, we generalize the results of the above section to a relaxed
control problem. The idea for relaxed the strict control problem defined above
is to embed the set $U$ of strict controls into a wider class which gives a
more suitable topological structure. In the relaxed model, the $U$-valued
process $v$ is replaced by a $\mathbb{P}\left(  U\right)  $-valued process
$q$, where $\mathbb{P}\left(  U\right)  $ denotes the space of probability
measure on $U$ equipped with the topology of stable convergence.

Let $V$ the set of positive random measures on $\left[  0,T\right]  \times U$
whose projection on $\left[  0,T\right]  $ coincide with the Lebesgue measure
$dt$. Equipped with the topology of stable convergence of measures, $V$ is a
compact metrizable space. The stable convergence is required for bounded
measurable functions $f\left(  t,a\right)  $ such that for each fixed
$t\in\left[  0,T\right]  $, $h\left(  t,.\right)  $ is continuous. The space
$V$ is equipped with its Borel $\sigma$-field, which is the smallest $\sigma
$-field such that the mapping $q\longmapsto%
{\displaystyle\int}
f\left(  s,a\right)  q\left(  ds,da\right)  $ are measurable for any bounded
measurable function $f$, continuous with respect to $a$ (Instead of functions
bounded and continuous with respect to the pair $\left(  t,a\right)  $ for the
weak topology)$.$

For more details, see Jacod-Memin $\left[  18,\ \text{page\ 629-630}\right]  $
and El Karoui et al $\left[  9,\ \text{Page 4-5}\right]  .$

\begin{definition}
A relaxed control $\left(  q_{t}\right)  _{t}$ is a $\mathbb{P}\left(
U\right)  $-valued process, progressively measurable with respect to $\left(
\mathcal{F}_{t}\right)  _{t}$\ and such that for each $t$, $1_{]0,t]}.q$\ is
$\mathcal{F}_{t}$-measurable.

\textit{We denote by }$\mathcal{R}$\textit{\ the set of all relaxed controls.}
\end{definition}

\ 

Every relaxed control $q$ may be desintegrated as $q\left(  dt,da\right)
=q\left(  t,da\right)  dt=q_{t}\left(  da\right)  dt$, where $q_{t}\left(
da\right)  $ is a progressively measurable process with value in the set of
probability measures $\mathbb{P}(U).$

The set $U$ is embedded into the set $\mathcal{R}$\ of relaxed process by the
mapping
\[
f:v\in U\mathbb{\longmapsto}f_{v}\left(  dt,da\right)  =\delta_{v_{t}%
}(da)dt\in\mathcal{R}%
\]
where $\delta_{v}$ is the atomic measure concentrated at a single point $v$.

\ 

For more details on relaxed controls, see $\left[  2\right]  ,\left[
4\right]  ,\left[  5\right]  ,\left[  12\right]  ,\left[  16\right]  ,\left[
26\right]  ,\left[  27\right]  .$

\ 

For any $q\in\mathcal{R}$, we consider the following relaxed BSDE
\begin{equation}
\left\{
\begin{array}
[c]{l}%
dy_{t}^{q}=\int_{U}b\left(  t,y_{t}^{q},z_{t}^{q},a\right)  q_{t}\left(
da\right)  dt+z_{t}^{q}dW_{t},\\
y_{T}^{q}=\xi.
\end{array}
\right.
\end{equation}

The expected cost associated to a relaxed control $q$ is defined as follows%
\begin{equation}
\mathcal{J}\left(  q\right)  =\mathbb{E}\left[  g\left(  y_{0}^{q}\right)
+\int_{0}^{T}\int_{U}h\left(  t,y_{t}^{q},z_{t}^{q},a\right)  q_{t}\left(
da\right)  dt\right]  .
\end{equation}

Our objective is to minimize the functional $\mathcal{J}$ over $\mathcal{R}$.
If $\mu\in\mathcal{R}$ is an optimal relaxed control, that is
\begin{equation}
\mathcal{J}\left(  \mu\right)  =\inf\limits_{q\in\mathcal{R}}\mathcal{J}%
\left(  q\right)  .
\end{equation}

\ 

Throughout this section we suppose moreover that
\begin{align}
&  U\text{ is compact,}\nonumber\\
&  b\text{ and }h\text{ are bounded,}\\
&  b_{y},h_{y},b_{z}\text{ and }h_{z}\text{ are Lipschitz continuous in
}z.\nonumber
\end{align}

\begin{remark}
If we put
\begin{align*}
\overline{b}\left(  t,y_{t}^{q},z_{t}^{q},q_{t}\right)   &  =\int_{U}b\left(
t,y_{t}^{q},z_{t}^{q},a\right)  q_{t}\left(  da\right)  ,\\
\overline{h}\left(  t,y_{t}^{q},z_{t}^{q},q_{t}\right)   &  =\int_{U}h\left(
t,y_{t}^{q},z_{t}^{q},a\right)  q_{t}\left(  da\right)  ,
\end{align*}
then equation $\left(  25\right)  $ becomes
\[
\left\{
\begin{array}
[c]{l}%
dy_{t}^{q}=\overline{b}\left(  t,y_{t}^{q},z_{t}^{q},q_{t}\right)
dt+z_{t}^{q}dW_{t},\\
y^{q}\left(  T\right)  =\xi,
\end{array}
\right.
\]
with a functional cost given by
\[
\mathcal{J}\left(  q\right)  =\mathbb{E}\left[  g\left(  y_{0}^{q}\right)
+\int_{0}^{T}\overline{h}\left(  t,y_{t}^{q},z_{t}^{q},q_{t}\right)
dt\right]  .
\]

Hence by introducing relaxed controls, we have replaced $U$ by a larger space
$\mathbb{P}\left(  U\right)  $. We have gained the advantage that
$\mathbb{P}\left(  U\right)  $ is both compact and convex, the new drift and
the integral coefficient of $\mathcal{J}$ are linear in $q.$

On the other hand, the coefficients $\overline{b}\ $(defined above) check the
same assumptions as $b$. Then, under assumptions $\left(  4\right)  $,
$\overline{b}\ $is uniformly Lipschitz and with linear growth. Then, by
classical results on BSDEs (The Pardoux-Peng theorem, see : Pardoux-Peng
$\left[  28\right]  $), for every $q\in\mathcal{R}$, equation $\left(
25\right)  $ has a unique solution.

Moreover, It is easy to see that $\overline{h}$ checks the same assumptions as
$h$. Then, the functional cost $\mathcal{J}$ is well defined from
$\mathcal{R}$ into $\mathbb{R}$.
\end{remark}

\begin{remark}
If $q_{t}=\delta_{v_{t}}$ is an atomic measure concentrated at a single point
$v_{t}$, then for each $t\in\left[  0,T\right]  $ we have
\begin{align*}
\int_{U}b\left(  t,y_{t}^{q},z_{t}^{q},a\right)  q_{t}\left(  da\right)   &
=\int_{U}b\left(  t,y_{t}^{q},z_{t}^{q},a\right)  \delta_{v_{t}}\left(
da\right)  =b\left(  t,y_{t}^{q},z_{t}^{q},v_{t}\right)  ,\\
\int_{U}h\left(  t,y_{t}^{q},z_{t}^{q},a\right)  q_{t}\left(  da\right)   &
=\int_{U}h\left(  t,y_{t}^{q},z_{t}^{q},a\right)  \delta_{v_{t}}\left(
da\right)  =h\left(  t,y_{t}^{q},z_{t}^{q},v_{t}\right)  .
\end{align*}

In this case $\left(  y^{q},z^{q}\right)  =\left(  y^{v},z^{v}\right)  $,
$J\left(  v\right)  =\mathcal{J}\left(  q\right)  $ and we get an ordinary
admissible control problem. So the problem of strict controls defined in the
section 2 is a particular case of the problem of relaxed one.
\end{remark}

\subsection{Approximation of trajectories}

The next lemma, known as the Chattering Lemma, tells us that any relaxed
control is a stable limit of a sequence of strict controls. This lemma was
first proved for deterministic measures and then extended to random measures
in $\left[  12\right]  $ and $\left[  16\right]  $.

\begin{lemma}
(Chattering\thinspace\thinspace Lemma).\thinspace\textit{Let }$q_{t}%
$\textit{\ be a predictable process with values in the space of probability
measures on }$U$\textit{. Then there exists a sequence of predictable
processes }$\left(  u^{n}\right)  _{n}$\textit{\ with values in }%
$U$\textit{\ such that }%
\begin{equation}
dtq_{t}^{n}\left(  da\right)  =dt\delta_{u_{t}^{n}}\left(  da\right)
\underset{n\longrightarrow\infty}{\longrightarrow}dtq_{t}\left(  da\right)
\text{ stably},\text{\textit{\ \ }}\mathcal{P}-a.s.
\end{equation}

\end{lemma}

\begin{proof}
See El Karoui et al $\left[  12\right]  .$
\end{proof}

\begin{lemma}
Let $q$ be a relaxed control and $\left(  u^{n}\right)  _{n}$ be a sequence of
strict controls such that $\left(  29\right)  $ holds. Then for any bounded
function $f:\left[  0,T\right]  \times U\rightarrow\mathbb{R}$, measurable in
$t$ and continuous in $a$, we have%
\begin{equation}%
{\displaystyle\int\nolimits_{U}}
f\left(  t,a\right)  \delta_{u_{t}^{n}}\left(  da\right)  \underset
{n\longrightarrow\infty}{\longrightarrow}%
{\displaystyle\int\nolimits_{U}}
f\left(  t,a\right)  q_{t}\left(  da\right)  .
\end{equation}

\end{lemma}

\begin{proof}
By the Chattering lemma and the definition of the stable convergence (see
Jacod-Memin $\left[  21,\ \text{definition 1.1, page 529}\right]  $, we have%
\[%
{\displaystyle\int\nolimits_{0}^{T}}
{\displaystyle\int\nolimits_{U}}
f\left(  t,a\right)  \delta_{u_{t}^{n}}\left(  da\right)  dt\underset
{n\longrightarrow\infty}{\longrightarrow}%
{\displaystyle\int\nolimits_{0}^{T}}
{\displaystyle\int\nolimits_{U}}
f\left(  t,a\right)  q_{t}\left(  da\right)  dt.
\]

Put%
\[
g\left(  s,a\right)  =1_{\left[  0,t\right]  }\left(  s\right)  f\left(
s,a\right)  .
\]

It's clear that%
\[%
{\displaystyle\int\nolimits_{0}^{T}}
{\displaystyle\int\nolimits_{U}}
g\left(  s,a\right)  \delta_{u_{s}^{n}}\left(  da\right)  ds\underset
{n\longrightarrow\infty}{\longrightarrow}%
{\displaystyle\int\nolimits_{0}^{T}}
{\displaystyle\int\nolimits_{U}}
g\left(  s,a\right)  q_{s}\left(  da\right)  ds.
\]

Then%
\[%
{\displaystyle\int\nolimits_{0}^{t}}
{\displaystyle\int\nolimits_{U}}
f\left(  s,a\right)  \delta_{u_{s}^{n}}\left(  da\right)  ds\underset
{n\longrightarrow\infty}{\longrightarrow}%
{\displaystyle\int\nolimits_{0}^{t}}
{\displaystyle\int\nolimits_{U}}
f\left(  s,a\right)  q_{s}\left(  da\right)  ds.
\]

The set $\left\{  \left(  s,t\right)  \ ;\ 0\leq s\leq t\leq T\right\}  $
generate $\mathcal{B}_{\left[  0,T\right]  }$. Then $\forall B\in
\mathcal{B}_{\left[  0,T\right]  }$, we have%
\[%
{\displaystyle\int\nolimits_{B}}
{\displaystyle\int\nolimits_{U}}
f\left(  s,a\right)  \delta_{u_{s}^{n}}\left(  da\right)  ds\underset
{n\longrightarrow\infty}{\longrightarrow}%
{\displaystyle\int\nolimits_{B}}
{\displaystyle\int\nolimits_{U}}
f\left(  s,a\right)  q_{s}\left(  da\right)  ds.
\]

This implies that%
\[%
{\displaystyle\int\nolimits_{U}}
f\left(  s,a\right)  \delta_{u_{s}^{n}}\left(  da\right)  \underset
{n\longrightarrow\infty}{\longrightarrow}%
{\displaystyle\int\nolimits_{U}}
f\left(  s,a\right)  q_{s}\left(  da\right)  \ ,\ \ dt-a.e.
\]

The lemma is proved.
\end{proof}

\ 

The next lemma gives the stability of the controlled stochastic differential
equation with respect to the control variable.

\begin{lemma}
Let $q_{t}\in\mathcal{R}$ be a relaxed control and $\left(  y^{q}%
,z^{q}\right)  $\textit{\ the corresponding trajectory. Then there exists a
sequence }$\left(  u^{n}\right)  _{n}\subset\mathcal{U}$ such that
\begin{equation}
\underset{n\rightarrow\infty}{\lim}\mathbb{E}\left[  \underset{t\in\left[
0,T\right]  }{\sup}\left\vert y_{t}^{n}-y_{t}^{q}\right\vert ^{2}\right]  =0,
\end{equation}%
\begin{equation}
\underset{n\rightarrow\infty}{\lim}\mathbb{E}\int_{0}^{T}\left\vert z_{t}%
^{n}-z_{t}^{q}\right\vert ^{2}dt=0,
\end{equation}%
\begin{equation}
\underset{n\rightarrow\infty}{\lim}J\left(  u^{n}\right)  =\mathcal{J}\left(
q\right)  ,
\end{equation}
where $\left(  y^{n},z^{n}\right)  $ denotes the solution of equation $\left(
1\right)  $ associated with $u^{n}$.
\end{lemma}

\begin{proof}
We have%
\begin{align*}
d\left(  y_{t}^{n}-y_{t}^{q}\right)   &  =\left[  b\left(  t,y_{t}^{n}%
,z_{t}^{n},u_{t}^{n}\right)  -b\left(  t,y_{t}^{q},z_{t}^{q},u_{t}^{n}\right)
\right]  dt\\
&  +\left[  b\left(  t,y_{t}^{q},z_{t}^{q},u_{t}^{n}\right)  -%
{\displaystyle\int\nolimits_{U}}
b\left(  t,y_{t}^{q},z_{t}^{q},a\right)  q_{t}\left(  da\right)  \right]  dt\\
&  +\left(  z_{t}^{n}-z_{t}^{q}\right)  dW_{t}%
\end{align*}

Put%
\begin{align*}
Y_{t}^{n}  &  =y_{t}^{n}-y_{t}^{q},\\
Z_{t}^{n}  &  =z_{t}^{n}-z_{t}^{q},
\end{align*}
and%
\begin{align}
\varphi^{n}\left(  t,Y_{t}^{n},Z_{t}^{n}\right)   &  =b\left(  t,y_{t}%
^{q},z_{t}^{q},u_{t}^{n}\right)  -%
{\displaystyle\int\nolimits_{U}}
b\left(  t,y_{t}^{q},z_{t}^{q},a\right)  q_{t}\left(  da\right) \\
&  +%
{\displaystyle\int\nolimits_{0}^{1}}
b_{y}\left(  t,y_{t}^{q}+\lambda\left(  y_{t}^{n}-y_{t}^{q}\right)  ,z_{t}%
^{q}+\lambda\left(  z_{t}^{n}-z_{t}^{q}\right)  ,u_{t}^{n}\right)  Y_{t}%
^{n}d\lambda\nonumber\\
&  +%
{\displaystyle\int\nolimits_{0}^{1}}
b_{z}\left(  t,y_{t}^{q}+\lambda\left(  y_{t}^{n}-y_{t}^{q}\right)  ,z_{t}%
^{q}+\lambda\left(  z_{t}^{n}-z_{t}^{q}\right)  ,u_{t}^{n}\right)  Z_{t}%
^{n}d\lambda.\nonumber
\end{align}

Then%
\begin{equation}
\left\{
\begin{array}
[c]{l}%
dY_{t}^{n}=\varphi^{n}\left(  t,Y_{t}^{n},Z_{t}^{n}\right)  dt+Z_{t}^{n}%
dW_{t},\\
Y_{T}^{n}=0.
\end{array}
\right.
\end{equation}

The above equation is a linear BSDE with bounded coefficients and with
terminal condition $Y_{T}^{n}=0$, then by applying a priori estimates (see
Briand et al $\left[  8\right]  $), we get%
\[
\mathbb{E}\left[  \underset{t\in\left[  0,T\right]  }{\sup}\left\vert
Y_{t}^{n}\right\vert ^{2}+%
{\displaystyle\int\nolimits_{0}^{T}}
\left\vert Z_{t}^{n}\right\vert ^{2}dt\right]  \leq C\mathbb{E}\left\vert
{\displaystyle\int\nolimits_{0}^{T}}
\left\vert \varphi^{n}\left(  t,0,0\right)  \right\vert dt\right\vert ^{2}.
\]

From $\left(  34\right)  $, we get%
\begin{align}
&  \mathbb{E}\left[  \underset{t\in\left[  0,T\right]  }{\sup}\left\vert
Y_{t}^{n}\right\vert ^{2}+%
{\displaystyle\int\nolimits_{0}^{T}}
\left\vert Z_{t}^{n}\right\vert ^{2}dt\right] \\
&  \leq C\mathbb{E}%
{\displaystyle\int\nolimits_{0}^{T}}
\left\vert
{\displaystyle\int\nolimits_{U}}
b\left(  t,y_{t}^{q},z_{t}^{q},a\right)  \delta_{u_{t}^{n}}\left(  da\right)
-%
{\displaystyle\int\nolimits_{U}}
b\left(  t,y_{t}^{q},z_{t}^{q},a\right)  q_{t}\left(  da\right)  \right\vert
^{2}dt.\nonumber
\end{align}

By $\left(  30\right)  $ and the dominated convergence theorem, the term in
the right hand side of the above inequality tends to zero as $n$ tends to
infinity. This prove $\left(  31\right)  $ and $\left(  32\right)  $.

\ 

Let us prove $\left(  33\right)  $

\ 

Since $g$\ and $h$\ are Lipshitz continuous in $\left(  y,z\right)  $, then by
using the Cauchy-Schwartz inequality, we have%
\begin{align*}
&  \left\vert J\left(  u^{n}\right)  -\mathcal{J}\left(  q\right)  \right\vert
\\
&  \leq C\left(  \mathbb{E}\left\vert y_{0}^{n}-y_{0}^{q}\right\vert
^{2}\right)  ^{1/2}+C\left(  \int_{0}^{T}\mathbb{E}\left\vert y_{t}^{n}%
-y_{t}^{q}\right\vert ^{2}ds\right)  ^{1/2}+C\left(  \mathbb{E}\int_{0}%
^{T}\left\vert z_{t}^{n}-z_{t}^{q}\right\vert ^{2}dt\right)  ^{1/2}\\
&  +C\left(  \mathbb{E}\int_{0}^{T}\left\vert
{\displaystyle\int\nolimits_{U}}
b\left(  t,y_{t}^{q},z_{t}^{q},a\right)  \delta_{u_{t}^{n}}\left(  da\right)
-\int_{U}h\left(  t,y_{t}^{q},z_{t}^{q},a\right)  q_{t}\left(  da\right)
\right\vert ^{2}dt\right)  ^{1/2}.
\end{align*}

From $\left(  31\right)  $ and $\left(  32\right)  $\ the first, the second
and the third terms in the right hand side converge to zero, and by $\left(
30\right)  $ and the dominated convergence theorem, the fourth term in the
right hand side tends to zero.
\end{proof}

\begin{remark}
As a consequence, it is easy to see that the strict and relaxed optimal
control problems have the same value function.
\end{remark}

\subsection{necessary optimality conditions for near controls}

In this section we derive necessary optimality conditions for near optimal
controls. This result is based on Ekeland's variational principle which is
given by the following.

\begin{lemma}
(Ekeland's variational principle). \textit{Let }$\left(  E,d\right)
$\textit{\ be a complete metric space and }$f:E\longrightarrow\overline
{\mathbb{R}}$\textit{\ be lower-semicontinuous and bounded from below. Given
}$\varepsilon>0$, suppose $u^{\varepsilon}\in E$ satisfies $f\left(
u^{\varepsilon}\right)  \leq\inf\left(  f\right)  +\varepsilon.$ Then for any
$\lambda>0$, \textit{there exists }$v\in E$\textit{\ such that }

\begin{enumerate}
\item $f\left(  v\right)  \leq f\left(  u^{\varepsilon}\right)  .$

\item $d\left(  u^{\varepsilon},v\right)  \leq\lambda.$

\item $f\left(  v\right)  <f\left(  w\right)  +\frac{\varepsilon}{\lambda
}d\left(  v,w\right)  \;,\;\forall\;w\neq v.$
\end{enumerate}
\end{lemma}

\begin{proof}
See Ekeland $\left[  10\right]  .$
\end{proof}

To apply Ekeland's variational principle, we have to endow the set
$\mathcal{U}$ of strict controls with an appropriate metric. For any
$u,v\in\mathcal{U}$, we set
\[
d\left(  u,v\right)  =\mathcal{P}\otimes dt\left\{  \left(  \omega,t\right)
\in\Omega\times\left[  0,T\right]  ,\;u\left(  t,\omega\right)  \neq v\left(
t,\omega\right)  \right\}  ,
\]
where $\mathcal{P}\otimes dt$ is the product measure of $\mathcal{P}$ with the
Lebesgue measure $dt$.

\ 

Let us summarize some of the properties satisfied by $d.$

\begin{lemma}
\begin{enumerate}
\item $\left(  \mathcal{U},d\right)  $\textit{\ is a complete metric space.}

\item The cost functional $J$ is continuous from $\mathcal{U}$ into
$\mathbb{R}$.
\end{enumerate}
\end{lemma}

\begin{proof}
See Mezerdi $\left[  25\right]  .$
\end{proof}

Now let $\mu\in\mathcal{R}$ be an optimal relaxed control and denote by
$\left(  y^{\mu},z^{\mu}\right)  $ the trajectory of the system controlled by
$\mu$. From lemmas $9,\ 10$ and 11, there exists a sequence $\left(
u^{n}\right)  _{n}$ of strict controls such that
\begin{align*}
dt\mu_{t}^{n}\left(  da\right)   &  =dt\delta_{u_{t}^{n}}\left(  da\right)
\underset{n\longrightarrow\infty}{\longrightarrow}dt\mu_{t}\left(  da\right)
\text{ Stably},\text{\textit{\ \ }}\mathcal{P}\text{-}a.s,\\
&  \mathbb{E}\left[  \underset{t\in\left[  0,T\right]  }{\sup}\left\vert
y_{t}^{n}-y_{t}^{\mu}\right\vert ^{2}\right]  \underset{n\longrightarrow
\infty}{\longrightarrow}0,\\
&  \mathbb{E}\int_{0}^{T}\left\vert z_{t}^{n}-z_{t}^{\mu}\right\vert
^{2}dt\underset{n\longrightarrow\infty}{\longrightarrow}0.
\end{align*}
where $\left(  y_{t}^{n},z_{t}^{n}\right)  $ is the solution of equation
$\left(  25\right)  $ controlled by $\mu^{n}.$

According to the optimality of $\mu$ and $\left(  29\right)  $, there exists a
sequence $\left(  \varepsilon_{n}\right)  _{n}$ of positive real numbers with
$\underset{n\rightarrow\infty}{\lim}\varepsilon_{n}=0$ such that
\[
J\left(  u^{n}\right)  =\mathcal{J}\left(  \mu^{n}\right)  \leq\mathcal{J}%
\left(  \mu\right)  +\varepsilon_{n}.
\]

A suitable version of lemma $13$ implies that, given any $\varepsilon_{n}>0$,
there exists $\left(  u^{n}\right)  _{n}\in\mathcal{U}$ such that
\begin{align}
J\left(  u^{n}\right)   &  \leq\underset{u\in\mathcal{U}}{\inf}J\left(
u\right)  +\varepsilon_{n},\nonumber\\
J\left(  u^{n}\right)   &  \leq J\left(  u\right)  +\varepsilon_{n}d\left(
u^{n},u\right)  \;;\;\forall u\in\mathcal{U}\text{.}%
\end{align}

Let us define the perturbation%
\begin{equation}
u_{t}^{n,\theta}=\left\{
\begin{array}
[c]{l}%
v\text{ \ if }t\in\left[  \tau,\tau+\theta\right]  ,\\
u_{t}^{n}\text{ \ Otherwise.}%
\end{array}
\right.
\end{equation}

From $\left(  37\right)  $ we have
\[
0\leq J\left(  u_{t}^{n,\theta}\right)  -J\left(  u^{n}\right)  +\varepsilon
_{n}d\left(  u^{n,\theta},u_{t}^{n}\right)  .
\]

From the definition of the metric $d,$ we obtain
\begin{equation}
0\leq J\left(  u_{t}^{n,\theta}\right)  -J\left(  u^{n}\right)  +\varepsilon
_{n}C\theta.
\end{equation}

From these above inequalities, we shall establish necessary optimality
conditions for near optimal controls.

\begin{theorem}
(Necessary optimality conditions for near controls). For each $\varepsilon
_{n}>0$\textit{, there exists }$\left(  u^{n}\right)  _{n}\in\mathcal{U}$
\textit{such that there exists a unique adapted processes}%
\[
p^{n}\in\mathcal{L}^{2}\left(  \left[  0,T\right]  ;\mathbb{R}^{n}\right)  ,
\]
\textit{solution of the following forward stochastic differential equation }%
\begin{equation}
\left\{
\begin{array}
[c]{l}%
-dp_{t}^{n}=H_{y}\left(  t,y_{t}^{n},z_{t}^{n},p_{t}^{n},u_{t}^{n}\right)
dt+H_{z}\left(  t,y_{t}^{n},z_{t}^{n},p_{t}^{n},u_{t}^{n}\right)  dW_{t},\\
p_{0}^{n}=g_{y}\left(  y_{0}^{n}\right)  ,
\end{array}
\right.
\end{equation}
\textit{such that for all }$v\in\mathcal{U}$,%
\begin{equation}
0\leq\left[  H\left(  t,y_{t}^{n},z_{t}^{n},p_{t}^{n},u_{t}^{n}\right)
-H\left(  t,y_{t}^{n},z_{t}^{n},p_{t}^{n},v\right)  \right]  +C\varepsilon
_{n}.
\end{equation}

\end{theorem}

\begin{proof}
From inequality $\left(  39\right)  $, we use the same method as in the last
sections with index $n$.
\end{proof}

\subsection{Necessary and sufficient optimality conditions for relaxed
controls}

In this subsection, we will state and prove necessary as well as sufficient
optimality conditions for relaxed controls. For this end, let us summarize and
prove some of lemmas that we will use thereafter.

\ 

Introduce the following adjoint equation in the relaxed form%
\begin{equation}
\left\{
\begin{array}
[c]{l}%
-dp_{t}^{\mu}=H_{y}^{\mu}\left(  t,y_{t}^{\mu},z_{t}^{\mu},p_{t}^{\mu},\mu
_{t}\right)  dt+H_{z}^{\mu}\left(  t,y_{t}^{\mu},z_{t}^{\mu},p_{t}^{\mu}%
,\mu_{t}\right)  dW_{t},\\
p_{0}^{\mu}=g_{y}\left(  y_{0}^{\mu}\right)  ,
\end{array}
\right.
\end{equation}%
\[
p^{\mu}\in\mathcal{L}^{2}\left(  \left[  0,T\right]  ;\mathbb{R}^{n}\right)
,
\]
where the Hamiltonian $H^{\mu}$ in the relaxed form is defined from $\left[
0,T\right]  \times\mathbb{R}^{n}\times\mathcal{M}_{n\times d}\left(
\mathbb{R}\right)  \times\mathbb{R}^{n}\times\mathbb{P}\left(  U\right)  $
into $\mathbb{R}$ by%
\[
H^{\mu}\left(  t,y_{t}^{\mu},z_{t}^{\mu},p_{t}^{\mu},\mu_{t}\right)
=p_{t}^{\mu}\int_{U}b\left(  t,y_{t}^{\mu},z_{t}^{\mu},a\right)  \mu
_{t}\left(  a\right)  -\int_{U}h\left(  t,y_{t}^{\mu},z_{t}^{\mu},a\right)
\mu_{t}\left(  a\right)  .
\]

For simplicity of notation, we denote
\begin{align*}
f^{n}(t)  &  =f\left(  t,y_{t}^{n},z_{t}^{n},u_{t}^{n}\right)  ,\\
f^{\mu}(t)  &  =\int_{U}\left(  t,y_{t}^{\mu},z_{t}^{\mu},a\right)  \mu
_{t}\left(  a\right)  ,
\end{align*}
\bigskip where $f$ stands for one of the functions $b_{y},b_{z},h_{y},h_{z}.$

\begin{lemma}
The following estimations hold%
\begin{align}
\underset{n\rightarrow\infty}{\lim}\mathbb{E}\int_{0}^{t}\left\vert b_{y}%
^{n}\left(  s\right)  -b_{y}^{\mu}\left(  s\right)  \right\vert ^{2}ds  &
=0,\\
\underset{n\rightarrow\infty}{\lim}\mathbb{E}\int_{0}^{t}\left\vert b_{z}%
^{n}\left(  s\right)  -b_{z}^{\mu}\left(  s\right)  \right\vert ^{2}ds  &
=0,\\
\underset{n\rightarrow\infty}{\lim}\mathbb{E}\int_{0}^{t}\left\vert h_{y}%
^{n}\left(  s\right)  -h_{y}^{\mu}\left(  s\right)  \right\vert ^{2}ds  &
=0,\\
\underset{n\rightarrow\infty}{\lim}\mathbb{E}\int_{0}^{t}\left\vert h_{z}%
^{n}\left(  s\right)  -h_{z}^{\mu}\left(  s\right)  \right\vert ^{2}ds  &  =0.
\end{align}

\end{lemma}

\begin{proof}
We have%
\begin{align*}
\mathbb{E}\int_{0}^{t}\left\vert b_{y}^{n}\left(  s\right)  -b_{y}^{\mu
}\left(  s\right)  \right\vert ^{2}ds  &  =\mathbb{E}\int_{0}^{t}\left\vert
b_{y}\left(  s,y_{s}^{n},z_{s}^{n},u_{s}^{n}\right)  -\int_{U}b_{y}\left(
s,y_{s}^{\mu},z_{s}^{\mu},a\right)  \mu_{s}\left(  a\right)  \right\vert
^{2}ds\\
&  \leq\mathbb{E}\int_{0}^{t}\left\vert b_{y}\left(  s,y_{s}^{n},z_{s}%
^{n},u_{s}^{n}\right)  -b_{y}\left(  s,y_{s}^{\mu},z_{s}^{n},u_{s}^{n}\right)
\right\vert ^{2}ds\\
&  +\mathbb{E}\int_{0}^{t}\left\vert b_{y}\left(  s,y_{s}^{\mu},z_{s}%
^{n},u_{s}^{n}\right)  -b_{y}\left(  s,y_{s}^{\mu},z_{s}^{\mu},u_{s}%
^{n}\right)  \right\vert ^{2}ds\\
&  +\mathbb{E}\int_{0}^{t}\left\vert b_{y}\left(  s,y_{s}^{\mu},z_{s}^{\mu
},u_{s}^{n}\right)  -\int_{U}b_{y}\left(  s,y_{s}^{\mu},z_{s}^{\mu},a\right)
\mu_{s}\left(  a\right)  \right\vert ^{2}ds.
\end{align*}

Since $b_{y}$ is Lipschitz continuous in $z$, then
\begin{align}
\mathbb{E}\int_{0}^{t}\left\vert b_{y}^{n}\left(  s\right)  -b_{y}^{\mu
}\left(  s\right)  \right\vert ^{2}ds  &  \leq\mathbb{E}\int_{0}^{t}\left\vert
b_{y}\left(  s,y_{s}^{n},z_{s}^{n},u_{s}^{n}\right)  -b_{y}\left(
s,y_{s}^{\mu},z_{s}^{n},u_{s}^{n}\right)  \right\vert ^{2}ds\nonumber\\
&  +C\mathbb{E}\int_{0}^{t}\left\vert z_{s}^{n}-z_{s}^{\mu}\right\vert
^{2}ds\\
&  +\mathbb{E}\int_{0}^{t}\left\vert b_{y}\left(  s,y_{s}^{\mu},z_{s}^{\mu
},u_{s}^{n}\right)  -\int_{U}b_{y}\left(  s,y_{s}^{\mu},z_{s}^{\mu},a\right)
\mu_{s}\left(  a\right)  \right\vert ^{2}ds.\nonumber
\end{align}

From $\left(  32\right)  $, we have%
\[
\underset{n\rightarrow\infty}{\lim}\mathbb{E}\int_{0}^{t}\left\vert z_{s}%
^{n}-z_{s}^{\mu}\right\vert ^{2}ds=0.
\]

Since $b_{y}$ is bounded and continuous, then by $\left(  31\right)  $ and the
dominate convergence theorem, we have%
\[
\underset{n\rightarrow\infty}{\lim}\mathbb{E}\int_{0}^{t}\left\vert
b_{y}\left(  s,y_{s}^{n},z_{s}^{n},u_{s}^{n}\right)  -b_{y}\left(
s,y_{s}^{\mu},z_{s}^{n},u_{s}^{n}\right)  \right\vert ^{2}ds=0.
\]

On the other hand, by the chattering lemma and the dominate convergence
theorem, we have%
\[
\underset{n\rightarrow\infty}{\lim}\mathbb{E}\int_{0}^{t}\left\vert \int
_{U}b_{y}\left(  s,y_{s}^{\mu},z_{s}^{\mu},a\right)  \delta_{u_{s}^{n}}\left(
da\right)  -\int_{U}b_{y}\left(  s,y_{s}^{\mu},z_{s}^{\mu},a\right)  \mu
_{s}\left(  a\right)  \right\vert ^{2}ds=0.
\]

By $\left(  47\right)  $ and these above three limits, we deduce $\left(
43\right)  $. Using the same method and arguments, we prove $\left(
44\right)  ,\left(  45\right)  $ and $\left(  46\right)  $.
\end{proof}

\begin{lemma}
Let $p^{n}$ and $p^{\mu}$ respectively the solutions of $\left(  40\right)  $
and $\left(  42\right)  $, then we have
\begin{equation}
\underset{n\rightarrow\infty}{\lim}\mathbb{E}\left[  \underset{t\in\left[
0,T\right]  }{\sup}\left\vert p_{t}^{n}-p_{t}^{\mu}\right\vert ^{2}\right]
=0.
\end{equation}

\end{lemma}

\begin{proof}
From $\left(  40\right)  $ and $\left(  42\right)  $, we have%
\begin{align*}
p_{t}^{n}  &  =g_{y}\left(  y_{0}^{n}\right)  -\int_{0}^{t}H_{y}^{n}\left(
s\right)  ds-\int_{0}^{t}H_{z}^{n}\left(  s\right)  dW_{s},\\
p_{t}^{\mu}  &  =g_{y}\left(  y_{0}^{\mu}\right)  -\int_{0}^{t}H_{y}^{\mu
}\left(  s\right)  ds-\int_{0}^{t}H_{z}^{\mu}\left(  s\right)  dW_{s},
\end{align*}
where%
\begin{align*}
H_{y}^{n}\left(  t\right)   &  =H_{y}\left(  t,y_{t}^{n},z_{t}^{n},p_{t}%
^{n},u_{t}^{n}\right)  \ \ ;\ \ H_{y}^{\mu}\left(  t\right)  =\int_{U}%
H_{y}\left(  t,y_{t}^{\mu},z_{t}^{\mu},p_{t}^{\mu},a\right)  \mu_{t}\left(
a\right)  ,\\
H_{z}^{n}\left(  t\right)   &  =H_{z}\left(  t,y_{t}^{n},z_{t}^{n},p_{t}%
^{n},u_{t}^{n}\right)  \ \ ;\ \ H_{z}^{\mu}\left(  t\right)  =\int_{U}%
H_{z}\left(  t,y_{t}^{\mu},z_{t}^{\mu},p_{t}^{\mu},a\right)  \mu_{t}\left(
a\right)  .
\end{align*}

Then%
\begin{align*}
\mathbb{E}\left\vert p_{t}^{n}-p_{t}^{\mu}\right\vert ^{2}  &  \leq
C\mathbb{E}\left\vert g_{y}\left(  y_{0}^{n}\right)  -g_{y}\left(  y_{0}^{\mu
}\right)  \right\vert ^{2}+C\mathbb{E}\int_{0}^{t}\left\vert H_{y}^{n}\left(
s\right)  -H_{y}^{\mu}\left(  s\right)  \right\vert ^{2}ds\\
&  +C\mathbb{E}\int_{0}^{t}\left\vert H_{z}^{n}\left(  s\right)  -H_{z}^{\mu
}\left(  s\right)  \right\vert ^{2}ds\\
&  \leq C\mathbb{E}\int_{0}^{t}\left\vert b_{y}^{n}\left(  s\right)  \left(
p_{s}^{n}-p_{s}^{\mu}\right)  \right\vert ^{2}ds+C\mathbb{E}\int_{0}%
^{t}\left\vert b_{z}^{n}\left(  s\right)  \left(  p_{s}^{n}-p_{s}^{\mu
}\right)  \right\vert ^{2}ds+C\alpha_{t}^{n},
\end{align*}
where
\begin{align}
\alpha_{t}^{n}  &  =\mathbb{E}\left\vert g_{y}\left(  y_{0}^{n}\right)
-g_{y}\left(  y_{0}^{\mu}\right)  \right\vert ^{2}+\mathbb{E}\int_{0}%
^{t}\left\vert h_{y}^{n}\left(  s\right)  -h_{y}^{\mu}\left(  s\right)
\right\vert ^{2}ds\\
&  +\mathbb{E}\int_{0}^{t}\left\vert \left(  b_{y}^{n}\left(  s\right)
-b_{y}^{\mu}\left(  s\right)  \right)  p_{s}^{\mu}\right\vert ^{2}%
ds+\mathbb{E}\int_{0}^{t}\left\vert h_{z}^{n}\left(  s\right)  -h_{z}^{\mu
}\left(  s\right)  \right\vert ^{2}ds\nonumber\\
&  +\mathbb{E}\int_{0}^{t}\left\vert \left(  b_{z}^{n}\left(  s\right)
-b_{z}^{\mu}\left(  s\right)  \right)  p_{s}^{\mu}\right\vert ^{2}ds.\nonumber
\end{align}

Since $b_{y}$ and $b_{z}$ are bounded then%
\begin{equation}
\mathbb{E}\left\vert p_{t}^{n}-p_{t}^{\mu}\right\vert ^{2}\leq2C\mathbb{E}%
\int_{0}^{t}\left\vert p_{s}^{n}-p_{s}^{\mu}\right\vert ^{2}ds+C\alpha_{t}%
^{n}.
\end{equation}

\ 

Let us prove that $\underset{n\rightarrow\infty}{\lim}\alpha_{t}^{n}=0$

\ 

Since $g_{y}$ is bounded and continuous, then by $\left(  31\right)  $ and the
dominated convergence theorem, we have%
\begin{equation}
\underset{n\rightarrow\infty}{\lim}\mathbb{E}\left\vert g_{y}\left(  y_{0}%
^{n}\right)  -g_{y}\left(  y_{0}^{\mu}\right)  \right\vert ^{2}=0.
\end{equation}

On the other hand, since $b_{y}$ is bounded, then%
\begin{equation}
\left\vert \left[  b_{y}^{n}\left(  s\right)  -b_{y}^{\mu}\left(  s\right)
\right]  p_{s}^{n}\right\vert \leq2C\left\vert p_{s}^{n}\right\vert .
\end{equation}

Hence by the Cauchy-Schwartz inequality we get,
\[
\mathbb{E}\int_{0}^{t}\left\vert \left[  b_{y}^{n}\left(  s\right)
-b_{y}^{\mu}\left(  s\right)  \right]  p_{s}^{\mu}\right\vert ds\leq\left(
\mathbb{E}\int_{0}^{t}\left\vert b_{y}^{n}\left(  s\right)  -b_{y}^{\mu
}\left(  s\right)  \right\vert ^{2}ds\right)  ^{1/2}\left(  \mathbb{E}\int
_{0}^{t}\left\vert p_{s}^{\mu}\right\vert ^{2}ds\right)  ^{1/2}.
\]

Since $p^{\mu}\in\mathcal{L}^{2}\left(  \left[  0,T\right]  ;\mathbb{R}%
^{n}\right)  $, then%
\[
\mathbb{E}\int_{0}^{t}\left\vert \left[  b_{y}^{n}\left(  s\right)
-b_{y}^{\mu}\left(  s\right)  \right]  p_{s}^{\mu}\right\vert ds\leq C\left(
\mathbb{E}\int_{0}^{t}\left\vert b_{y}^{n}\left(  s\right)  -b_{y}^{\mu
}\left(  s\right)  \right\vert ^{2}ds\right)  ^{1/2}.
\]

By $\left(  43\right)  $, we have
\[
\underset{n\rightarrow\infty}{\lim}\mathbb{E}\int_{0}^{t}\left\vert b_{y}%
^{n}\left(  s\right)  -b_{y}^{\mu}\left(  s\right)  \right\vert ^{2}ds=0.
\]

Then, we deduce that%
\begin{equation}
\underset{n\rightarrow\infty}{\lim}\mathbb{E}\int_{0}^{t}\left\vert \left[
b_{y}^{n}\left(  s\right)  -b_{y}^{\mu}\left(  s\right)  \right]  p_{s}^{\mu
}\right\vert ds=0.
\end{equation}

By using the dominated convergence theorem we obtain
\begin{equation}
\underset{n\rightarrow\infty}{\lim}\mathbb{E}\int_{0}^{t}\left\vert \left[
b_{y}^{n}\left(  s\right)  -b_{y}^{\mu}\left(  s\right)  \right]  p_{s}^{\mu
}\right\vert ^{2}ds=0.
\end{equation}

Similarly, using $\left(  44\right)  $, the boundeness of $b_{z}$ and the
dominated convergence theorem, it follows that%
\begin{equation}
\underset{n\rightarrow\infty}{\lim}\mathbb{E}\int_{0}^{t}\left\vert \left[
b_{z}^{n}\left(  s\right)  -b_{z}^{\mu}\left(  s\right)  \right]  p_{s}^{\mu
}\right\vert ^{2}ds=0.
\end{equation}

From $\left(  45\right)  ,\left(  46\right)  ,\left(  51\right)  ,\left(
54\right)  $ and $\left(  55\right)  $, it is easy to see that
\begin{equation}
\underset{n\rightarrow\infty}{\lim}\alpha_{t}^{n}=0.
\end{equation}

Finally from $\left(  50\right)  ,\left(  56\right)  $, Gronwall's lemma and
Bukholder-Davis-Gundy inequality, we have the desired result$.$
\end{proof}

\ 

\begin{theorem}
(Necessary optimality conditions for relaxed controls). \textit{Let }$\mu$
\textit{be an optimal relaxed control minimizing the cost }$\mathcal{J}%
$\textit{\ over }$\mathcal{R}$ \textit{and }$\left(  y_{t}^{\mu},z_{t}^{\mu
}\right)  $\textit{\ the corresponding optimal trajectory. Then there exists a
unique adapted processes }
\[
p^{\mu}\in\mathcal{L}^{2}\left(  \left[  0,T\right]  ;\mathbb{R}^{n}\right)
,
\]
\textit{solution of the stochastic forward differential equation }$\left(
42\right)  $, \textit{such that for all} $q\in\mathcal{R}$, we have
\begin{equation}
H^{\mu}\left(  t,y_{t}^{\mu},z_{t}^{\mu},p_{t}^{\mu},\mu_{t}\right)
=\underset{q\in\mathbb{P}\left(  U\right)  }{\max}H^{\mu}\left(  t,y_{t}^{\mu
},z_{t}^{\mu},p_{t}^{\mu},q\right)  .
\end{equation}

\end{theorem}

\begin{proof}
Let $\mu$ be an optimal relaxed control. By the necessary condition for near
controls (Theorem 15), there exists a sequence $\left(  u^{n}\right)
_{n}\subset\mathcal{U}$ such that \textit{for all} $v\in\mathcal{U}$%
\[
0\leq\left[  H\left(  t,y_{t}^{n},z_{t}^{n},p_{t}^{n},u_{t}^{n}\right)
-H\left(  t,y_{t}^{n},z_{t}^{n},p_{t}^{n},v\right)  \right]  +C\varepsilon
_{n},
\]
where $\underset{n\rightarrow\infty}{\lim}\varepsilon_{n}=0.$

According to $\left(  29\right)  ,\left(  31\right)  ,\left(  31\right)  $ and
$\left(  48\right)  $, the result follows immediately by letting $n$ going to
infinity in the last inequality.
\end{proof}

\begin{remark}
If $\mu_{t}\left(  da\right)  =\delta_{u\left(  t\right)  }\left(  da\right)
$, we recover the strict necessary optimality conditions (Theorem 4)$.$
\end{remark}

\begin{theorem}
(Sufficient optimality conditions for relaxed controls). We know that the set
$\mathcal{R}$ of relaxed controls is convex and the function $H^{q}\left(
t,y_{t}^{q},z_{t}^{q},p_{t}^{q},q_{t}\right)  $ is linear in $q_{t}$ . If we
assume that for every $q\in\mathcal{R}$ and for all $t\in\left[  0,T\right]
$, the functions $g$ is convex and $\left(  y_{t}^{q},z_{t}^{q}\right)
\longrightarrow H^{q}\left(  t,y_{t}^{q},z_{t}^{q},p_{t}^{q},q_{t}\right)  $
is concave, then $\mu$ is an optimal relaxed control if it satisfies $\left(
57\right)  .$
\end{theorem}

\begin{proof}
The proof is the same that in theorem 5.
\end{proof}

\end{document}